\documentclass[11pt]{amsart}

\usepackage{graphicx}
\usepackage[labelsep=space]{caption}
\usepackage{amsfonts}
\usepackage{tikz-cd}
\usepackage{amsthm}
\usepackage{nccmath}
\usepackage{etex}
\usepackage{amssymb,mathtools}
\usepackage{mathrsfs}
\usepackage{dsfont}
\usepackage[all]{xy}
\usepackage{microtype}

\usepackage[utf8]{inputenc}
\usepackage[T1]{fontenc}
\usepackage{lmodern} 
\usepackage{latexsym}
\usepackage{MnSymbol}
\usepackage{nicefrac}
\usepackage{color}
\usepackage{blindtext}
\usepackage{enumitem}
\usepackage{old-arrows}
\usepackage{empheq}
\usepackage{extpfeil}
\usepackage{hyperref}
\usepackage{theoremref}
\usepackage{float}
\usepackage{subcaption}

\makeatletter
\renewenvironment{proof}[1][\proofname]{%
	\par\pushQED{\qed}\normalfont%
	\topsep6\p@\@plus6\p@\relax
	\trivlist\item[\hskip\labelsep\bfseries#1\@addpunct{.}]%
	\ignorespaces
}{%
	\popQED\endtrivlist\@endpefalse
}
\makeatother

\usepackage{geometry}
\usepackage{secdot}

\usepackage{ifxetex}
\usepackage{etoolbox}
\ifxetex
\usepackage{unicode-math}
\makeatletter
\patchcmd{\arrowfill@}{-7mu}{-14mu}{}{}
\patchcmd{\arrowfill@}{-7mu}{-14mu}{}{}
\patchcmd{\arrowfill@}{-2mu}{-4mu}{}{}
\patchcmd{\arrowfill@}{-2mu}{-4mu}{}{}
\makeatother
\fi

\theoremstyle{plain}
\newtheorem{theorem}{Theorem}[section]
\newtheorem{lemma}[theorem]{Lemma}

\newtheorem{proposition}[theorem]{Proposition}
\theoremstyle{remark}
\newtheorem{remark}[theorem]{Remark}
\theoremstyle{definition}
\newtheorem{definition}[theorem]{Definition}
\newtheorem{example}[theorem]{Example}

\newcommand{\Homeo}{\mathrm{Homeo}}
\newcommand{\PHomeo}{\mathrm{PHomeo}}
\newcommand{\MCG}{\mathrm{MCG}}
\newcommand{\PMCG}{\mathrm{PMCG}}
\newcommand{\LB}{\mathrm{LB}}
\newcommand{\PLB}{\mathrm{PLB}}
\newcommand{\Aut}{\mathrm{Aut}}

\usetikzlibrary{decorations.markings}
\tikzset{double line with arrow/.style args={#1,#2}{decorate,decoration={markings,%
			mark=at position 0 with {\coordinate (ta-base-1) at (0,1pt);
				\coordinate (ta-base-2) at (0,-1pt);},
			mark=at position 1 with {\draw[#1] (ta-base-1) -- (0,1pt);
				\draw[#2] (ta-base-2) -- (0,-1pt);
}}}}

\newtheorem{theoremx}{Theorem}

\begin{document}
	
\title{The generalized Lefschetz number and loop braid groups}
\author{Stavroula Makri}
\address{Department of Mathematics, Vrije Universiteit Amsterdam, Faculty of Science, De Boelelaan 1111, 1081 HV Amsterdam, the Netherlands}
\email{s.makri@vu.nl}
	
\begin{abstract}
We study the interplay between braid group theory and topological dynamics in three dimensions.
While classical braid theory has been extensively applied to surface homeomorphisms to analyze
fixed and periodic points, an analogous framework in three-dimensional manifolds has been lacking.

In this work, we use loop braid groups, a three-dimensional generalization of classical braid
groups, to investigate homeomorphisms of the $3$-ball that leave invariant a trivial link of
finitely many circles. In our main theorem we relate matrix representations of loop braid
groups, which are related to the Burau representation, to the abelianized generalized Lefschetz
number. This result provides information on the existence of fixed points of such
homeomorphisms and on their linking numbers with the invariant link. As an application, we
obtain a lower bound for the number of periodic points of a given period. Our theorem
establishes a three-dimensional analogue of a classical result for braids in the disk, and
connects loop braid groups with Nielsen fixed point theory and topological dynamics in dimension
three.
\end{abstract}

	\maketitle
\section{Introduction}

The connection between braid group theory and topological dynamics arises from the
applications of braid group theory to the study of the periodic orbit structure of iterated surface homeomorphisms.
In dynamical systems, topological invariants are often used to study the qualitative and quantitative properties of
the system. Elements of braid groups are such invariants, characterizing the topological behavior of periodic orbits
in the case of surface homeomorphisms isotopic to the identity. Let $S$ be a compact surface and $f:S \to S$ a homeomorphism isotopic to the identity map.
Let $P$ be a finite invariant set of $f$, that is, a union of finitely many periodic orbits of $f$.
Choose and fix an isotopy $\{f_t\}_{0 \le t \le 1}$ from $\mathrm{id}$ to $f$.
The motion $f_t(P)$ of the points of $P$ defines a braid on $S$. The use of braids in surface dynamics began in the 1980s, and we refer the reader to \cite{Boyland1994Topological, matsuoka} for a detailed overview of topological dynamics on surfaces and braid theory.

One of the main interactions between braid group theory and surface dynamics is the use of matrix representations of braids in order to obtain information about the existence of fixed points of the homeomorphism $f$ associated with the given braid. More precisely, for $S = \mathbb{D}^2$, the trace of the Burau matrix of the braid is closely related to the abelianized generalized Lefschetz number of $f$ on $S \setminus P$, as established by Fadell, Husseini, and Fried in the early 1980s. The generalized Lefschetz number records the Nielsen fixed point classes of $f$ together with their indices, and its abelianization records how the corresponding fixed points are linked with $P$. Thus, by computing the matrix representation one obtains the existence of fixed points and periodic orbits of $f$ in $S \setminus P$, together with their linking numbers with $P$, as stated in Theorem 3.3 in \cite{matsuoka} for the case of the punctured $2$-disk. Similar results have been obtained for the torus by Huang and Jiang in \cite{huangJiang1989} and for the annulus by Jiang \cite{Jiang1992Periodic}. The main objective of this work is to establish a three-dimensional analogue of this result using loop braid groups.

Braid theory has been used to study a number of dynamical properties of surface homeomorphisms \cite{LlibreMacKay1990, FranksMisiurewicz1993, Boyland1994Topological, Guaschi1994PseudoAnosov, Los1997Forcing, deCarvalhoHall2004}, and a natural question is whether similar techniques can be used to study fixed points and periodic orbits on $3$-dimensional manifolds. While a considerable amount of research has been carried out on the complexity of orbit structures in dynamical systems on $3$-dimensional manifolds \cite{Epstein1972, Rechtman2010, CristofaroGardiner2023}, there are no known results on $3$-dimensional manifolds analogous to the results on surfaces that use braid theory. One of the main reasons is that the classical braid group is trivial on $3$-dimensional manifolds,
and therefore the connection between braid theory and topological dynamics does not extend
directly beyond dimension two.

In this article we propose to use instead the loop braid groups $\LB_n$, a three-dimensional analogue of the Artin braid groups $B_n$ in which points on a surface are replaced by circles in the $3$-ball. The groups $\LB_n$ appear in the literature under many different names: conjugating automorphisms of the free group $F_n$, due to Savushkina \cite{savushkina1996group}, loop braid groups, due to Baez--Crans--Wise \cite{baez2007exotic}, and groups of untwisted rings, due to Brendle--Hatcher \cite{brendle2013configuration}. Their algebraic and topological properties have been studied extensively \cite{goldsmith1981theory, fenn1997braid, Vershinin2001, brendle2013configuration, damiani2017journey, Bellingeri2020, representationLoop, DarnePalmerSoulie2025}. There are several interpretations of $\LB_n$, which will be presented in Section~\ref{loop}. That is, as the fundamental group of a space of configurations of circles, as a group of automorphisms of the free group $F_n$, and as a mapping class group.

Let $C_n = S^1_1 \sqcup \dots \sqcup S^1_n$ be a trivial link of $n$ oriented components in the interior of the $3$-ball $B^3$. We denote by $\operatorname{Homeo}(B^3, C_n)$ the group of orientation-preserving homeomorphisms $f$ of $B^3$ that fix $\partial B^3$ pointwise, have the property that $f(C_n) = C_n$, and preserve the orientation of each component of $C_n$. Every such $f$ is isotopic to the identity relative to $\partial B^3$. Choose an isotopy $\{f_t\}_{0 \le t \le 1}$ from $\mathrm{id}$ to $f$ relative to $\partial B^3$.
The motion $f_t(C_n)$ of the circles of $C_n$ starts and ends at $C_n$, and defines an
element of $\LB_n$, which we denote by $b_{f, C_n}$.

In this work, we develop a three-dimensional counterpart of the classical correspondence between braid group representations and generalized Lefschetz numbers, relating representations of loop braid groups to the abelianization of the generalized Lefschetz number for homeomorphisms in $\operatorname{Homeo}(B^3, C_n)$. In the two-dimensional case, this correspondence was presented initially by Fadell and Husseini \cite{fadell1983} as an application of Nielsen fixed point theory, then appeared explicitly as a theorem by Matsuoka in \cite{Ma1}, and was given with a detailed proof by Huang and Jiang \cite{huangJiang1989}. In our setting, we consider the space $B = B^3 \setminus C_n$ and, since we are interested in compact spaces, we consider its compactification $\bar{B}$, obtained by replacing each circle $S^1_i$ by a boundary torus, together with the induced homeomorphism $\bar{f} : \bar{B} \to \bar{B}$, which coincides with $f$ on $B$. These objects are explicitly defined in Subsection \ref{compact}.

To state our results, we introduce the following notation. Let $b = b_{f, C_n}$, and let $\mu \in S_n$ be the permutation induced by $f$ on the components of $C_n$, that is, $f(S^1_i) = S^1_{\mu(i)}$. Let $\mu_1, \dots, \mu_m$ be the cycles of $\mu$, and let $P_j \subseteq C_n$ be the union of the components $S^1_i$ with $i \in \mu_j$. To each fixed point $x$ of $\bar{f}$ we associate its linking numbers $\ell k_{\bar{f}}(P_j, x) \in \mathbb{Z}$, $j = 1, \dots, m$, which measure how the trajectory of $x$ under the isotopy winds around $P_j$. The linking numbers are defined in Subsection \ref{linking}. For $I = (i_1, \dots, i_m) \in \mathbb{Z}^m$, we denote by $\operatorname{Fix}_I(\bar{f})$ the set of fixed points of $\bar{f}$ whose linking numbers with $P_1, \dots, P_m$ are $i_1, \dots, i_m$, and by $\operatorname{ind}(\operatorname{Fix}_I(\bar{f}))$ its fixed point index. Finally, let $\Lambda = \mathbb{Z}[a_1^{\pm 1}, \dots, a_n^{\pm 1}]$, and let $R, \bar{R} : \LB_n \to GL_n(\Lambda)$ be the twisted representations of $\LB_n$ defined in Subsection \ref{matrixR}, which are related to the Burau representation of $\LB_n$. We denote by $\pi_\mu : \Lambda \to \mathbb{Z}[t_1^{\pm 1}, \dots, t_m^{\pm 1}]$ the ring homomorphism sending $a_i$ to $t_j$ for $i \in \mu_j$, and by $R(b)^{\pi_\mu}$ the matrix obtained by applying $\pi_\mu$ to each entry of $R(b)$, and similarly for $\bar{R}(b)$. The theorem below is our main result, which appears in Section \ref{Main} as Theorem \ref{thmmain}.

\begin{theoremx}
Let $f \in \operatorname{Homeo}(B^3, C_n)$ and $b = b_{f,C_n}$. Then
\[
1+\operatorname{tr}(S(b)^{\pi_\mu}) = \sum_{I \in \mathbb{Z}^m} \operatorname{ind}\big(\operatorname{Fix}_I(\bar{f})\big) \, t_1^{i_1} \cdots t_m^{i_m} \in \mathbb{Z}[t_1^{\pm 1}, \dots, t_m^{\pm 1}],
\]
where $S(b)^{\pi_\mu}=\bar{R}(b)^{\pi_\mu}-R(b)^{\pi_\mu}$.
\end{theoremx}

The right-hand side of the equality is the abelianization of the generalized Lefschetz number of $\bar{f}$, which is defined in Subsection \ref{abelian}. Each monomial $t_1^{i_1} \cdots t_m^{i_m}$ with non-zero coefficient corresponds to at least one fixed point $x$ of $\bar{f}$ whose linking numbers with $P_1, \dots, P_m$ are $i_1, \dots, i_m$. Thus, by computing the trace of a matrix associated with the loop braid $b$, we obtain the existence of fixed points of $\bar{f}$, together with their linking numbers with the components of $C_n$.

As an application of Theorem~\ref{thmmain}, we obtain a lower bound for the number of periodic points of $f$. Let $\operatorname{Per}_{p, C_n}(f)$ be the set of periodic points of $f$ of minimal period $p$ that do not lie on $C_n$ and, for simplicity, let $Q_p:=1 - \operatorname{tr}\big((R(b)^{\pi_\mu})^p\big) + \operatorname{tr}\big((\bar{R}(b)^{\pi_\mu})^p\big)$. The theorem below appears in Section~\ref{Application} as Theorem~\ref{thmapp}.

\begin{theoremx}
Let $f \in \operatorname{Homeo}(B^3, C_n)$, let $p \in \mathbb{N}^*$, and assume that $f^p$ has no fixed points on $C_n$. Then
\[
|\operatorname{Per}_{p, C_n}(f)| \ge p\, M_p,
\]
where $M_p$ is the number of monomials $t_1^{i_1} \cdots t_m^{i_m}$ with non-zero coefficient in
$Q_p$
such that $\gcd(p, i_1, \dots, i_m) = 1$.
\end{theoremx}

The paper has the following structure. In Section~\ref{Lefschetz}, we recall the necessary background on Nielsen fixed point theory and the generalized Lefschetz number. In Section~\ref{loop}, we present the loop braid groups and the tools that we will need. In Section~\ref{Main}, we introduce the twisted representations $R$ and $\bar{R}$ of $\LB_n$, define the linking number of a fixed point of $f$ with a subset of $C_n$, and prove our main result, Theorem~\ref{thmmain}. We conclude with Section~\ref{Application}, where we illustrate Theorem~\ref{thmmain} with two examples and derive a lower bound for the number of periodic points of $f$.

\section{Generalized Lefschetz number}\label{Lefschetz}
  
In this section we give the necessary prerequisites about Nielsen fixed point classes and we define the generalized Lefschetz number. For further insights into Nielsen fixed point theory and the Lefschetz number we refer the reader to \cite{Brown1971Lefschetz, jiang1983} and \cite[Chapter IV]{JezierskiMarzantowicz2006}.

Let $X$ be a connected finite cell complex, and $f:X\to X$ a continuous map. We denote by $\text{Fix}(f)$ the fixed point set of $f$, that is, $\text{Fix}(f)=\{x\in X\ | \ f(x)=x\}$. 
We recall that two fixed points $x, y$ of $f$ are said to be \emph{Nielsen equivalent} if
there exists a path $\ell$ in $X$ from $x$ to $y$ such that $\ell$ and its image $f \circ \ell$ are homotopic relative to their endpoints. The equivalence class of a point $x \in \mathrm{Fix}(f)$ is called the \emph{Nielsen class} of $x$. We denote by $\mathrm{NC}(f)$ the set of Nielsen classes of $f$. A Nielsen class is \emph{essential} if it has nonzero fixed point index, as defined in (\cite{jiang1983}, page 17) and the \emph{Nielsen number} $N(f)$ is defined as the number of essential Nielsen classes of $f$.

In general, determining whether two fixed points belong to the same Nielsen class is a difficult problem and Reidemeister showed that this geometric problem can be translated into an algebraic one. Choose a base point $x_0 \in X$ and a path $w$ from $x_0$ to $f(x_0)$, called \emph{base path}. For simplicity, let $\pi = \pi_1(X, x_0)$ be the fundamental group of $X$ based at $x_0$. We define $f_\pi: \pi \to \pi$ as the composition of the induced map $f_*: \pi_1(X, x_0) \to \pi_1(X, f(x_0))$ with $w_*: \pi_1(X, f(x_0)) \to \pi_1(X, x_0)$. Two elements $\lambda_1, \lambda_2 \in \pi$ are said to be \emph{Reidemeister equivalent} if there exists $\lambda \in \pi$ such that
\[
\lambda_2 = f_\pi(\lambda) \cdot \lambda_1 \cdot \lambda^{-1}.
\]
An equivalence class of elements of $\pi$ under this relation is called a \emph{Reidemeister class} and we denote by $ \mathcal{R}(f_\pi)$ the set of all Reidemeister classes. 

We are ready to assign a Reidemeister class to a fixed point of $f$. For $x \in \mathrm{Fix}(f)$, choose a path $\ell$ from $x_0$ to $x$. The Reidemeister class of $x$ is the Reidemeister class of $w(f \circ \ell)\ell^{-1}\in \pi$ and will be denoted by $R(x)$. Note that the Reidemeister class represented by $[w(f \circ \ell)\ell^{-1}]$
is independent of the choice of $\ell$, since for another path $\ell'$ from $x_0$ to $x$, the loops $w(f \circ \ell)\ell^{-1}$ and $w(f \circ \ell')(\ell')^{-1}$  have the same Reidemeister class. All points in a Nielsen class $F$ share the same Reidemeister class, which is denoted by $R(F)$ and called the Reidemeister class of $F$ or \emph{coordinate of $F$} as in \cite{Jiang1996PeriodicOrbits}. Two fixed points belong to the same Nielsen class if and only if they have the same Reidemeister class. Thus, $\mathrm{NC}(f)$ can be viewed as a subset of $ \mathcal{R}(f_\pi)$.

For $\alpha \in \mathcal{R}(f_\pi)$, define $\mathrm{Fix}_\alpha(f) = \{x \in \mathrm{Fix}(f) \ | \ R(x) = \alpha\}$. Then, $$\text{Fix}(f)=\bigcup_{\alpha\in \mathcal{R}(f_\pi)} \mathrm{Fix}_\alpha(f).$$
If $\mathrm{Fix}_\alpha(f)$ is non-empty, it forms a Nielsen class with Reidemeister class $\alpha$. Note that the compactness of $X$ implies that $\mathrm{Fix}_\alpha(f)$ is empty except for finitely many $\alpha$. For an isolated set $S$ of fixed points of $f$, we denote by $\text{ind}(S)$ the fixed point index of $S$ with respect to $f$.
Moreover, we denote by $\mathbb{Z}\mathcal{R}(f_\pi)$ the free abelian group generated by the elements of $\mathcal{R}(f_\pi)$.

We can now define the generalized Lefschetz number or Reidemeister trace.

\begin{definition}
The \emph{generalized Lefschetz number} $\mathcal{L}(f)$, also called the \emph{Reidemeister trace}, of $f$ is defined by
\begin{equation}
\label{eq:ReidemeisterTrace}
\mathcal{L}(f) = \sum_{F \in \mathrm{NC}(f)} \mathrm{ind}(F) \cdot R(F)
       = \sum_{\alpha \in \mathcal{R}(f_\pi)} \mathrm{ind}(\mathrm{Fix}_\alpha(f)) \, \alpha \in \mathbb{Z}\mathcal{R}(f_\pi).
\end{equation}
\end{definition}

In \cite{kelly}, Kelly presents an algorithm to compute $\mathcal{L}(f)$ for orientation
preserving homeomorphisms on surfaces with boundary. We remark that the generalized Lefschetz number is a homotopy invariant. Let $g:X\to X$ be a continuous map homotopic to $f$ through a homotopy $\{h_t\}_{0\leq t\leq 1}$. We consider as base path for $g$ the composition of $w$ with the path $h_t(x_0)$, for $0\leq t \leq 1$, so that we have $f_\pi=g_\pi$. From Nielsen fixed point theory we deduce that $\mathcal{L}(f)=\mathcal{L}(g)$.

From the generalized Lefschetz number we can deduce the Nielsen number as well as the  classical Lefschetz number. The Nielsen number $N(f)$ is the number of Reidemeister classes with non-zero coefficients in $\mathcal{L}(f)$ and the classical Lefschetz number is equal to the sum of the coefficients in $\mathcal{L}(f)$. 

We will now present a trace formula for the generalized Lefschetz number, using the universal covering space of $X$. Let $\widetilde{X}$ be the universal covering space of $X$. The cell structure on $X$ induces a corresponding cell structure on $\widetilde{X}$. We denote by $C_q(\widetilde{X})$ the group of cellular $q$-chains over $\mathbb{Z}$. 
The action of $\pi$ on $\tilde{X}$ induces an action of the group ring $\mathbb{Z}\pi$ on $C_q(\widetilde{X})$. Thus, the group $C_q(\widetilde{X})$ becomes a free, finitely generated $\mathbb{Z}\pi$-module. Assume that $f: X \to X$ is a cellular map, and let $\widetilde{f}: \widetilde{X} \to \widetilde{X}$ be a lift of $f$. Then $\widetilde{f}$ induces the homomorphism
\[
\widetilde{f}_q: C_q(\widetilde{X}) \to C_q(\widetilde{X}).
\] 
The map $\widetilde{f}_q$ is an \emph{$f_\pi$-twisted homomorphism}, meaning that
\begin{equation}
\widetilde{f}_q(\lambda u) = f_\pi(\lambda) \, \widetilde{f}_q(u), \quad \lambda \in \pi, \, u \in C_q(\widetilde{X}).
\label{action}
\end{equation} 
We can define now the trace $\mathrm{tr}\, \widetilde{f}_q \in \mathbb{Z}\mathcal{R}(f_\pi)$ in the following way. Let $\tilde{e}_1, \dots, \tilde{e}_n$ be the $q$-cells of $\widetilde{X}$ forming a basis of $C_q(\widetilde{X})$ over $\mathbb{Z}\pi$ and $A = (a_{ij})$ the matrix over $\mathbb{Z}\pi$ representing the $f_\pi$-homomorphism $\widetilde{f}_q$ with respect to this basis. That is, the entries $a_{ij}$ are given by
\[
\widetilde{f}_q(\tilde{e}_i) = \sum_j a_{ij} \tilde{e}_j.
\]  
Then, we define the trace of $\tilde{f}_q$ as follows:
\[
\mathrm{tr}\, \widetilde{f}_q = [\mathrm{tr}\, A],
\]  
where $[\,\cdot\,]$ denotes the projection from $\mathbb{Z}\pi$ to $\mathbb{Z}\mathcal{R}(f_\pi)$.  

We note that by choosing another basis of $C_q(\tilde{X})$, the new basis will be related to the previous basis via a matrix $C$, such that the matrix representing $\tilde{f}_q$ with respect to the new basis will be given by $C^{f_\pi}AC^{-1}$, where $C^{f_\pi}$ presents the matrix $C$ with entries replaced by their $f_\pi$-images. It follows that the trace $\mathrm{tr}\,\tilde{f}_q$ is well defined, since $[\mathrm{tr}\, (C^{f_\pi}AC^{-1})]=[\mathrm{tr}\, A]\in \mathbb{Z}\mathcal{R}(f_\pi)$.  

In \cite{husseini1982}, Husseini proved the following trace formula for the generalized Lefschetz number:
\begin{equation}
\label{trace}
\mathcal{L}(f) = \sum_q (-1)^q \, \mathrm{tr}\, \widetilde{f}_q \in \mathbb{Z}\mathcal{R}(f_\pi).
\end{equation}
We will use this trace formula to prove our main theorem in Section \ref{Main}.

\section{Loop braid groups} \label{loop}

Let $n\in \mathbb{N}$. The loop braid groups $\LB_n$ appear in the literature under many different names: conjugating automorphisms of the free group $F_n$ due to Savushkina, \cite{savushkina1996group}, welded braid groups due to Fenn--Rim\'{a}nyi--Rourke, \cite{fenn1997braid} and groups of untwisted rings due to Brendle--Hatcher, \cite{brendle2013configuration}. The term ``loop braid groups'', which we will be using and denoting by $\LB_n$, was introduced by Baez--Crans--Wise, \cite{baez2007exotic}.

The groups $\LB_n$ are actually a $3$-dimensional analogue of the braid groups $B_n$, introduced by Artin \cite{artin1947braids}, and there are several interpretations: in terms of mapping class groups, fundamental group of a space of configurations of circles and automorphisms of the free group on $n$ generators $F_n$. We refer the reader to \cite{damiani2017journey} for a complete presentation of the equivalent definitions of the loop braid groups.

In this work we will mainly treat loop braid elements in terms of
mapping classes of a $3$-ball with $n$ circles in its interior that are left setwise invariant. Let $M$ be a compact, connected, orientable manifold, possibly with boundary, and
$N$ an orientable submanifold contained in the interior of $M$. A self-homeomorphism of the pair of manifolds $(M,N)$ is a homeomorphism
$f : M \to M$ that fixes $\partial M$ pointwise and fixes $N$ globally. Every self-homeomorphism of $(M,N)$ induces a permutation on the
connected components of $N$ in the natural way. We denote by $\Homeo(M,N)$ the group of self-homeomorphisms of $(M,N)$ that preserve
orientation on both $M$ and $N$. The multiplication in $\Homeo(M,N)$ is given by the
usual composition.
Moreover, we denote by $\PHomeo(M,N)$, where $\mathrm{P}$ stands for pure, the subgroup of self-homeomorphisms of $(M,N)$
that send each connected component of $N$ to itself.

We say that $f, g \in \Homeo(M, N)$ are isotopic relative to $N$ if there exists
an isotopy $h_t \in \Homeo (M, N)$ from $f$ to $g$. 

\begin{definition}
The \emph{mapping class group} of $M$ with respect to $N$,
denoted by $\MCG(M,N)$, is the group of isotopy classes of self-homeomorphisms in
$\Homeo(M,N)$, with multiplication determined by composition. 
The \emph{pure mapping class group} of $M$ with respect to $N$, denoted by $\PMCG(M,N)$, is the subgroup of elements of
$\MCG(M,N)$ that send each connected component of $N$ to itself.
\end{definition}

We will give now the definition of the loop braid group in terms of a mapping class group.
Fix $n \geq 1$, and let $C_n = S^1_1 \sqcup \cdots \sqcup S^1_n$ be a collection of $n$ disjoint,
unknotted, oriented circles forming a trivial link of $n$ components in $\mathbb{R}^3$.
The exact position of $C_n$ is irrelevant and thus we assume that
$C_n$ is contained in the $xy$-plane inside the $3$-ball $B^3$.
\begin{definition}
The \emph{loop braid group} on $n$ components, denoted by $\LB_n$, is the mapping class
group $\MCG(B^3,C_n)$. The \emph{pure loop braid group} on $n$ components, denoted by
$\PLB_n$, is the pure mapping class group $\PMCG(B^3,C_n)$.

\end{definition}

This definition is the three-dimensional counterpart of the classical
identification $B_n \cong \MCG(\mathbb{D}^2, Q_n)$, where $Q_n \subset \mathbb{D}^2$
is a set of $n$ points. As a braid may be viewed as a motion of $n$ points
in the disk, an element of $\LB_n$ may be viewed as a motion of $n$ circles in
the $3$-ball. More precisely, given $[f] \in \LB_n = \MCG(B^3, C_n)$, choose an
isotopy $\{f_t\}_{t \in [0,1]}$ of $B^3$ from $\mathrm{id}$ to $f$, which exists
by the Alexander trick. Then following $C_n$ under the isotopy $f_t$ gives a path of trivial links
starting and ending at $C_n$. That is, a motion of the circles of $C_n$ inside
$B^3$, which gives rise to a geometric description of the loop braid groups.

To give a geometric interpretation of the loop braid group $\LB_n$ we define the untwisted ring group, given by Brendle--Hatcher in \cite{brendle2013configuration}, which is actually isomorphic to $\LB_n$. We note that the following definition appears in \cite{goldsmith1981theory} under the name of motion groups.

\begin{definition}
Let $n\in \mathbb{N}$ and let $\mathcal{UR}_n$ be the space of all configurations of $n$ disjoint pairwise unlinked unordered Euclidean circles in $\mathbb{R}^3$ lying in planes parallel to a fixed one. The untwisted ring group $UR_n$ is its fundamental group.

\end{definition}
 
\begin{theorem}[Brendle--Hatcher, \cite{brendle2013configuration}]\label{bre}
Let $n\in \mathbb{N}$. The untwisted ring group $UR_n$ admits the following presentation:
$$\langle \sigma_1, \dots, \sigma_{n-1}, \rho_1, \dots, \rho_{n-1}\ |\ R  \rangle,$$
where $R$ is the set of the following relations:
 \begin{enumerate}[label=\textup{\Roman*.}]
 \item $\sigma_i\sigma_{i+1}\sigma_i=\sigma_{i+1}\sigma_i\sigma_{i+1}$, for $i=1, \dots, n-2$,
 \item $\sigma_i\sigma_j=\sigma_j\sigma_i$, for $|i-j|>1$,  where $1\leq i,j\leq n-1$,
 \item $\rho_i\rho_{i+1}\rho_i=\rho_{i+1}\rho_i\rho_{i+1}$, for $i=1, \dots, n-2$, 
 \item $\rho_i\rho_j=\rho_j\rho_i$, for $|i-j|>1$, where $1\leq i,j\leq n-1$, 
 \item $\rho_i^2=1$, for $i=1, \dots, n-1$,
 \item $\sigma_i\rho_j=\rho_j\sigma_i$, for $|i-j|>1$,  where $1\leq i,j\leq n-1$,
 \item $\sigma_i\rho_{i+1}\rho_i=\rho_{i+1}\rho_i\sigma_{i+1}$, for $i=1, \dots, n-2$,
 \item $\rho_i\sigma_{i+1}\sigma_i=\sigma_{i+1}\sigma_i\rho_{i+1}$, for $i=1, \dots, n-2$.
	\end{enumerate}
\end{theorem}

\begin{remark}
In Theorem \ref{bre} we see that relations (I.)-(II.) are the standard braid group relations, relations (III.)-(V.) are permutation group relations and relations (VI.)-(VIII.) are the so-called mixed relations.
\end{remark}

The generators $\sigma_i, \rho_i$, were initially considered by Goldsmith in \cite{goldsmith1981theory}. The generator $\sigma_i$ permutes the $i^{th}$ and the $(i+1)^{th}$ circles by passing the $i^{th}$ circle through the $(i+1)^{th}$ and the generator $\rho_i$ permutes them passing the $i^{th}$ around the $(i+1)^{th}$, as depicted in the following figure.

\begin{figure}[H]
\centering
\begin{tikzpicture}[
    scale=1.2,
    every path/.style={line width=0.8pt},
    halo/.style={preaction={draw, white, line width=4pt}},
    >={To[length=1.6mm, width=2.6mm]}
  ]
  \draw (0,0) ellipse (0.75 and 0.45);
  \node at (0,-0.8) {$i$};
  \draw (0.85,0.05) .. controls (1.2,-0.45) and (2.0,-0.35) .. (2.75,-0.05);
  \draw[halo] (3,0) ellipse (0.75 and 0.45);
  \node at (3,-0.8) {$i+1$};
  \draw[halo, ->] (2.75,-0.05) .. controls (3.15,0.1) and (3.6,0.45) .. (4.45,0.05);
  \node at (1.6,0.25) {$\sigma_i$};

  \begin{scope}[xshift=6.5cm]
    \draw (0,0) ellipse (0.75 and 0.45);
    \node at (0,-0.8) {$i$};
    \draw (3,0) ellipse (0.75 and 0.45);
    \node at (3,-0.8) {$i+1$};
        \draw[halo, ->] (0.95,0.05) .. controls (2.2,0.35) and (3.4,-0.25) .. (4.55,0.05);
    \node at (1.55,-0.15) {$\rho_i$};
  \end{scope}
\end{tikzpicture}
\caption{The generators $\sigma_i$ and $\rho_i$.}
\label{figure1}
\end{figure}
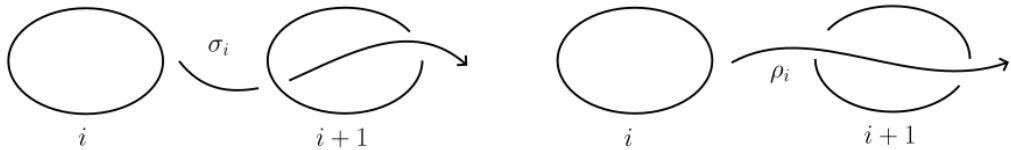

\begin{remark}\label{remarknotation}
     From \cite{damiani2017journey}, we know that the groups $\LB_n$ and $UR_n$ are actually equivalent formulations of the same group. This result follows as a consequence of a result that was established for more general groups in \cite{goldsmith1981theory} by Goldsmith. Thus, Theorem \ref{bre} provides a presentation of $\LB_n$. Since $\LB_n\cong UR_n$, we will henceforth, for simplicity, continue to use the notation $\LB_n$. 
 \end{remark}

As already mentioned, the loop braid group can also be described as a group of
automorphisms of the free group. This description comes from the action of
homeomorphisms on the fundamental group of the complement of $C_n$. We fix a base
point $z \in \partial B^3$ and we identify the fundamental group $\pi_1(B^3 \setminus C_n, z)$ with the free
group $F_n = \langle x_1, \dots, x_n \rangle$, where $x_i$ is represented by a loop
going once around $S^1_i$, that is, a meridian of $S^1_i$, joined to $z$ by a fixed arc.
    Every $f \in \Homeo(B^3, C_n)$ fixes $z$, and hence induces an automorphism $f_\pi$ of
    $\pi_1(B^3 \setminus C_n, z) \cong F_n$. By results of Dahm \cite{Dahm} and Goldsmith
    \cite{goldsmith1981theory}, this gives an injective map $\LB_n \to \Aut(F_n)$, which sends the
    generators $\sigma_i$ and $\rho_i$ of $\LB_n$ to the automorphisms $\sigma_i$ and $\rho_i$ of the
    following definition. The following algebraic
description of $\LB_n$, will play an important role in assigning a matrix to
each generator of $\LB_n$ in Section \ref{Main}.

\begin{definition}\label{automorphism}
Let $n\in \mathbb{N}$ and let $F_n$ denote the free group of rank $n$ with generators $\{x_1, \dots, x_n\}$ and let $\Aut(F_n)$ denote its
automorphism group. The group $\LB_n$ is the subgroup of $\Aut(F_n)$ that consists of the so-called conjugating automorphisms, $$a:x_i\mapsto W_ix_{\pi(i)}W_i^{-1},$$ where $\pi$ is a permutation and $W_i$ is a word in $F_n$. Let $\sigma_i\in \Aut(F_n),\ i=1,\dots,n-1$ be given by

\begin{equation}
\sigma_i:
\begin{cases}
x_i\mapsto x_ix_{i+1}x_i^{-1}\\
x_{i+1}\mapsto x_{i}\\
x_j\mapsto x_j,\ j\neq i, i+1
\end{cases}
\end{equation}
and let $\rho_i\in \Aut(F_n),\ i=1,\dots,n-1$ be given by

\begin{equation}
\rho_i:
\begin{cases}
x_i\mapsto x_{i+1}\\
x_{i+1}\mapsto x_i\\
x_j\mapsto x_j,\ j\neq i, i+1
\end{cases}
\end{equation}
The loop braid group $\LB_n$ is generated by both sets of elements $\sigma_i$ and $\rho_i$, $i=1, \dots, n-1$.
\end{definition}

\begin{remark}
The elements $\sigma_i$, $i=1,\dots,n-1$ generate the braid subgroup of $\Aut(F_n)$, which is isomorphic to the braid group $B_n$. Moreover, the elements $\rho_i$, $i=1,\dots, n-1$ generate the permutation subgroup of $\Aut(F_n)$ which is a copy of the symmetric group $S_n$.
\end{remark}

Before concluding this section, we present the unreduced Burau representation of $\LB_n$. We recall that the classical braid groups $B_n$ can be defined as the fundamental groups of the configuration spaces of points in the plane and we know that $B_n$ is isomorphic to the group of isotopy classes of homeomorphisms of the $2$-disk that act by the identity on its boundary and that leave invariant a subset $Q_n$ of $n$ points in its interior.
One of the oldest interesting representations of $B_n$ 
is the Burau representation which was defined by assigning precise matrices to the standard generators of $B_n$, but as a homological representation it admits a more natural interpretation.
The unreduced Burau representation of $B_n$ has been extended to $\LB_n$ by Vershinin in \cite{Vershinin2001}. For more insights into the representation of loop braid groups, we refer the reader to \cite{representationLoop}, where Palmer and Soulié present a topological construction of the Burau representations of the loop braid groups.

The matrices for the unreduced Burau representation $\LB_n \to GL_n(\mathbb{Z}[t^{\pm 1}])$
are given by
\[
\rho_i \mapsto I_{i-1} \oplus 
\begin{pmatrix}
0 & 1 \\
1 & 0
\end{pmatrix}
\oplus I_{n-i-1}
\]
and
\[
\sigma_i \mapsto I_{i-1} \oplus 
\begin{pmatrix}
1 - t & t \\
1 & 0
\end{pmatrix}
\oplus I_{n-i-1},
\]
where $I_k$ denotes the identity matrix of size $k$.
We note that the Burau representation is sometimes defined using the transposes of
these matrices.
In the following section we will see how this representation is related to the abelianized generalized Lefschetz number.
    
	\section{Main result}\label{Main}
  
In this section, we prove the main result of this work. Before doing so, we first establish the necessary setting. We are interested in isotopy classes of homeomorphisms of the $3$-ball that leave invariant a trivial link $C_n$ of $n$ components in its interior. That is, we consider the elements $f\in\Homeo(B^3,C_n)$ up to relative isotopy. As explained in Section \ref{loop}, any such element $f$ is expressed as a word in the $\sigma_i$'s and $\rho_j$'s, by Remark \ref{remarknotation}. The aim of this section is to show how the Burau representation of an element in $\LB_n$, that corresponds to a given $f$, gives us information about the existence of fixed points of $f$ and in addition their linking behavior, which we will define in a while, with the trivial link $C_n$.

\subsection{Twisted representation of $\LB_n$}\label{matrixR}
We shall start by introducing the \emph{twisted representation} of $\LB_n$.

Given a ring $A$, let $GL_k(A)$
be the group of invertible $k \times k$ matrices over $A$.
 A ring homomorphism $\nu : A \to A'$ to another ring $A'$ induces a homomorphism $GL_k(A) \to GL_k(A')$, which for simplicity we also denote by $\nu$, which replaces each entry of the matrix with its image under $\nu$.
For $R\in GL_k(A)$, its image $\nu(R) \in GL_k(A')$ will be denoted by $R^\nu$.

Let $\Lambda = \mathbb{Z}[a_1^{\pm1}, \dots, a_n^{\pm1}]$ be the ring of Laurent
polynomials over $\mathbb{Z}$ in the variables $a_1, \dots, a_n$, and let $I_j$
denote the identity matrix of size $j$. For $i = 1, \dots, n-1$, we define the
following matrices in $GL_n(\Lambda)$:
\begin{align*}
R_{\sigma_i} &= I_{i-1} \oplus
\begin{pmatrix} 1-a_{i+1} & a_i \\ 1 & 0 \end{pmatrix}
\oplus I_{n-i-1}, &
R_{\rho_i} &= I_{i-1} \oplus
\begin{pmatrix} 0 & 1 \\ 1 & 0 \end{pmatrix}
\oplus I_{n-i-1}
\end{align*}
and
\begin{align*}
    \bar{R}_{\sigma_i} &= I_{i-1} \oplus
\begin{pmatrix} 0 & a_i \\ 1 & 1-a_i \end{pmatrix}
\oplus I_{n-i-1}, &
\bar{R}_{\rho_i} &:= R_{\rho_i}.
\end{align*}
We then define the maps
$$
R : \LB_n \to GL_n(\Lambda), \quad \sigma_i \mapsto R_{\sigma_i}, \
\rho_i \mapsto R_{\rho_i}$$
\text{and}
$$
\bar{R} : \LB_n \to GL_n(\Lambda), \quad \sigma_i \mapsto \bar{R}_{\sigma_i}, \
\rho_i \mapsto \bar{R}_{\rho_i}.
$$
As we will see later, $R$ corresponds to an abelianized version of the action of
$\tilde{f}_1$ on $C_1(\tilde{X})$, and $\bar{R}$ corresponds to an abelianized version of the action of
$\tilde{f}_2$ on $C_2(\tilde{X})$, where these actions are defined in
Subsection \ref{universal}.

We call these two maps $R$ and $\bar{R}$ \emph{twisted representations} of $\LB_n$ because they are defined by the following rule, which coincides with the fact that $\tilde{f}_q$ is an $f_\pi$-twisted homomorphism, as expressed in \eqref{action}.
To an element $b \in \LB_n$ we can always associate an element of the symmetric group $S_n$, which actually corresponds to the induced permutation on the set of the trivial link $C_n$ of $n$ components. Thus, let $b \in \LB_n$ with permutation $\mu\in S_n$, and let
$\nu(b) : \Lambda \to \Lambda$
be the ring automorphism defined by
$\nu(b)(a_i) = a_{\mu(i)}.$ That is, the permutation $\mu$ acts on the indices of the variable $a_i$.
The maps $R : \LB_n \to GL_n(\Lambda)$ and $\bar{R} : \LB_n \to GL_n(\Lambda)$
are then defined by the rule
\begin{equation}
\label{product}
R(bb') = R(b)^{\nu(b')} \, R(b') 
\quad \text{and} \quad 
\bar{R}(bb') = \bar{R}(b)^{\nu(b')} \, \bar{R}(b').
\end{equation}

To sum up, we have associated to each generator $\sigma_i$ and $\rho_i$ two twisted representations $R, \bar{R}$, which, as we will see later, correspond to $\tilde{f}_1, \tilde{f}_2$ respectively, after having abelianized the entries of the matrices.

\subsection{Compactification}\label{compact}

As in the case of surfaces, it is often useful to use a blow-up procedure to replace invariant sets with boundary components, allowing us to study the dynamics on a compact space. This is achieved through a procedure similar to the one introduced by Bowen \cite{bowen}, which we adapt here to a $3$-dimensional setting.

In our setting, the loop braid group $\LB_n$ is initially defined as the topological mapping class group $\pi_0(\operatorname{Homeo}(B^3, C_n))$, where $C_n := S^1_1 \sqcup \dots \sqcup S^1_n$ is a trivial link of $n$ disjoint, unknotted, oriented circles in the interior of the $3$-ball $B^3$. It follows from two results of Wattenberg \cite[Lemma 1.4 and Lemma 2.4]{wattenberg} that the topological mapping class group of the 3-space with respect
to a collection of $n$ disjoint, unknotted, oriented circles, that form a trivial link, is
isomorphic to the $C^{\infty}$ mapping class group, defined in terms of diffeomorphisms.
Thus, we have that the following isomorphism holds:
\[
\pi_0(\operatorname{Homeo}(B^3, C_n)) \cong \pi_0(\operatorname{Diffeo}(B^3, C_n)).
\]
This isomorphism guarantees that every topological mapping class of the pair $(B^3, C_n)$ contains a smooth representative, and that topological isotopies can be smoothed. Consequently, without loss of generality, we can choose to represent any element of the loop braid group by a diffeomorphism $f$ that leaves the link $C_n$ invariant setwise. This smoothness is crucial as it justifies the use of differentials in the following geometric construction.

Equivalently, we consider the dynamics of the diffeomorphism $f$ restricted to the open manifold $B = B^3 \setminus C_n$. Because we are interested in compact spaces, we define its compactification $\bar{B}$ via a blow-up along $C_n$. For each circle $S^1_i$ of $C_n$, we consider its normal bundle $NS_i^1 = T B^3|_{S^1_i} / T S^1_i$. The blow-up replaces each $S^1_i$ with its unit normal bundle, which is homeomorphic to a torus $\mathbb{T}_i$. Topologically, this corresponds to deleting the open tubular neighborhoods of $C_n$ and retaining their boundary tori. We can therefore define the compactified space as $\bar{B} = (B^3 \setminus C_n) \cup T$, where $T := \bigsqcup_{i=1}^n \mathbb{T}_i$. Thus, $\bar{B}$ is a compact manifold with boundary, where a torus $\mathbb{T}_i$ appears as a boundary component for each deleted circle $S^1_i$.

Since $f$ is a diffeomorphism with $f(C_n) = C_n$, its restriction to
$B^3 \setminus C_n$ extends to a diffeomorphism $\bar{f} : \bar{B} \to \bar{B}$,
defined as follows. On $\bar B\setminus T = B^3\setminus C_n$ we set $\bar f = f$. On the
boundary, note that $f$ maps each circle $S^1_i$ to some circle $S^1_j$. Thus, for
$p \in S^1_i$ the differential $df_p$ maps $T_p S^1_i$ to $T_{f(p)} S^1_j$ and
therefore induces a linear isomorphism
$\overline{df_p} : N_p S^1_i \to N_{f(p)} S^1_j$. A point of $\mathbb{T}_i$ is
a pair $(p,u)$ with $p \in S^1_i$ and $u \in N_p S^1_i$, $\|u\| = 1$, and we set
\[
\bar{f}(p,u) = \left( f(p), \frac{\overline{df_p}(u)}{\|\overline{df_p}(u)\|} \right).
\]
This is well defined because $\overline{df_p}$ is invertible, so
$\overline{df_p}(u) \neq 0$.

\begin{remark}\label{extra}
Since $\bar{f}$ is equal to $f$ on $\bar{B} \setminus T = B^3 \setminus C_n$, the fixed points of $\bar{f}$
outside $T$ are exactly the points of $\operatorname{Fix}(f) \setminus C_n$. We note that $\bar{f}$ may
have additional fixed points on the boundary $T$, and these are determined by the maps
$\overline{df_q}$, $q \in C_n$. Specifically, a point $(q,u) \in \mathbb{T}_i$ is fixed by $\bar{f}$
if and only if $f(q) = q$ and $\overline{df_q}(u) = \lambda u$ for some $\lambda > 0$, that is, $u$ is
the direction of an eigenvector of $\overline{df_q} : N_q S^1_i \to N_q S^1_i$ with positive
eigenvalue. At the same time, a fixed point $q$ of $f$ on $C_n$ may give no fixed point of $\bar{f}$ if $\overline{df_q}$
has no positive real eigenvalue, for instance, if $\overline{df_q}$ is a non-trivial rotation.
Therefore, if $\overline{df_q}$ has no positive real eigenvalue for any
$q \in \operatorname{Fix}(f) \cap C_n$, then $\bar{f}$ has no fixed points on $T$, and
\[
\operatorname{Fix}(\bar{f}) = \operatorname{Fix}(f) \setminus C_n .
\]
In particular, if $C_n$ contains no fixed points of $f$, then
$\operatorname{Fix}(\bar{f}) = \operatorname{Fix}(f)$.
\end{remark}

\subsection{Linking number}\label{linking}
 We recall that homeomorphisms $f$ in $\mathrm{Homeo}(B^3,C_n)$ preserve
the orientation of $B^3$ and the orientation of each component of $C_n$. Every
$f$ permutes the components of $C_n$ and we denote by $\mu$ the induced permutation in the symmetric group $S_n$.
In this subsection we define the linking number of a fixed point $x_0\in\mathrm{Fix}(f)$ with a subset $P\subseteq C_n$ that is invariant under $f$, that is, $f(P)=P$.

For $i = 1, \dots, n$, fix pairwise disjoint closed tubular neighborhoods
$N(S^1_i) \cong S^1 \times D^2$ of the components $S^1_i$ of $C_n$, contained in
$\mathrm{Int}(B^3)$, with $S^1_i$ corresponding to $S^1 \times \{0\}$. Fix a point $p \in S^1$ and let
$w_i = \{p\} \times \partial D^2 \subset \partial N(S^1_i)$, the boundary of the
meridian disk $\{p\} \times D^2$, be oriented in such a way that its linking number with the
oriented circle $S^1_i$ is $+1$. Let $m_i$ denote its class in
$H_1(B^3 \setminus C_n)$. Applying the Mayer--Vietoris sequence to $B^3=(B^3\setminus C_n)\cup N$, where $N$ is an open tubular neighborhood of $C_n$, that is, a disjoint union of solid tori, one can show that $H_1(B^3 \setminus C_n)$ is freely generated by $m_1,\dots,m_n$. Thus, $H_1(B^3 \setminus C_n)\cong\mathbb{Z}^n$. 

Let $b \in \LB_n$ be represented by $f \in \mathrm{Homeo}(B^3, C_n)$, and let
$\{f_t\}_{t \in [0,1]}$ be an isotopy in $\mathrm{Homeo}(B^3)$ relative to $\partial B^3$ from $\mathrm{id}$ to
$f$, so that the motion $f_t(C_n)$ represents $b$. Consider the space
\[
Y := (B^3 \times [0,1]) \setminus \bigcup_{t \in [0,1]} \big(f_t(C_n) \times \{t\}\big).
\]
The map $\Psi : Y \to B^3 \setminus C_n$ defined by $\Psi(y,t) = f_t^{-1}(y)$
is a homotopy equivalence and moreover, we have
$\Psi(x,0) = x$ because $f_0 = \mathrm{id}$. Hence, $H_1(Y)\cong H_1(B^3 \setminus C_n)\cong\mathbb{Z}^n$ is free abelian with basis
the meridians at level $t=0$, that is, the classes of $w_i \times \{0\}$, which we
denote again by $m_1, \dots, m_n$.

We can now define the linking number of a fixed point $x_0\in\mathrm{Fix}(f)$ with an invariant subset $P\subseteq C_n$ of $f$.

\begin{definition}
Fix a point $z$ on $\partial B^3$. Let $x_0 \in \mathrm{Fix}(f) \setminus C_n$ and fix a path $\delta$ in
$B^3 \setminus C_n$ from $x_0$ to $z$. Let $\hat{\gamma}_{x_0}$
be the loop in $X$ obtained by concatenating the following four paths:
the path $\gamma_{x_0}(t) = (f_t(x_0), t)$ from $(x_0,0)$ to $(x_0,1)$,
the path $\delta \times \{1\}$, the segment $s_z=\{z\} \times [0,1]$ from $(z,1)$ to $(z,0)$ and the inverse of the path $\delta \times \{0\}$. That is,
\[
\hat{\gamma}_{x_0} := \gamma_{x_0} \cdot (\delta \times \{1\}) \cdot s_z \cdot
(\delta^{-1} \times \{0\}).
\]
Since $\hat{\gamma}_{x_0}$ is a loop in $X$, it defines a class in the homology group $H_1(X) \cong \mathbb{Z}^n$. Thus, we can uniquely write:
$$[\hat{\gamma}_{x_0}] = \sum_{i=1}^n l_i m_i,\ \text{for}\ l_i\in \mathbb{Z},
$$
where the meridian $m_i$ corresponds to the circle $S^1_i$ of $C_n$. We define the \emph{linking number} of $x_0$ with an $f$-invariant subset $P \subseteq C_n$ to be the following integer:
$$\ell k_f(P, x_0) := \sum_{S^1_i \in P} l_i\in \mathbb{Z}.$$
\end{definition}

\begin{remark}
Since $f$ fixes $\partial B^3$ pointwise, the map $\Psi(y,t) = f_t^{-1}(y)$ sends
$\hat{\gamma}_{x_0}$ to a loop homotopic to $f^{-1}(\delta) \cdot \delta^{-1}$. Thus
$[\hat{\gamma}_{x_0}] = \big[f^{-1}(\delta) \cdot \delta^{-1}\big] \quad
\text{in } H_1(Y) \cong H_1(B^3 \setminus C_n)$. In particular this class does not depend on $f_t$.  Although $\hat{\gamma}_{x_0}$ is defined using the
motion $f_t$, its class therefore depends only on $f$. Moreover, if $\delta'$ is another path from $x_0$ to $z$, then for the loop $c = \delta' \cdot \delta^{-1}$ at $x_0$ we have $\delta' \simeq c
\cdot \delta$. Thus, the class
$[\hat{\gamma}_{x_0}]$ becomes $f^{-1}_*[c]+[f^{-1}(\delta)\cdot\delta^{-1}]-[c]$, where $f^{-1}_*$ is the automorphism of $H_1(B^3 \setminus C_n)$ induced by $f^{-1}$. It follows that $[\hat{\gamma}_{x_0}]$ changes by $f^{-1}_*[c] - [c]$. Let $[c] = \sum_{i=1}^n l_i m_i$. Since $f^{-1}_*(m_i) = m_{\mu^{-1}(i)}$, we get,
after indexing again,
$$f^{-1}_*[c] = \sum_{i=1}^n l_i\, m_{\mu^{-1}(i)} = \sum_{i=1}^n l_{\mu(i)}\, m_i .$$
Since $f(P) = P$, the permutation
$\mu$ maps the set of indices of the components of $P$ to itself, and so
$$
\sum_{S^1_i \in P} l_{\mu(i)} = \sum_{S^1_i \in P} l_i.$$
Thus, the linking number $\ell k_f(P, x_0)$ does not change, which implies that it is independent of the choice of the path $\delta$.
\end{remark}

Given $f$, we extend the notion of linking number to the fixed points of the
blow-up $\bar{f} : \bar{B} \to \bar{B}$ as follows. The inclusion of
$B^3 \setminus C_n$ into $\bar{B}$ is a homotopy equivalence, so $H_1(\bar{B})$ is
free abelian with basis $m_1, \dots, m_n$. Let $x_0 \in \mathrm{Fix}(\bar{f})$ and fix a point $z \in \partial B^3$. We choose a path $\delta$ in $\bar{B}$ from $x_0$ to a point
$z$. Thus, the path
$\bar{f}^{-1}(\delta) \cdot \delta^{-1}$ is a loop based at $x_0$. For an
$f$-invariant subset $P \subseteq C_n$, we define $\ell k_{\bar{f}}(P, x_0)$ to be the
sum, over the components of $P$, of the coefficients of
$[\bar{f}^{-1}(\delta) \cdot \delta^{-1}]$ in the basis $m_1, \dots, m_n$ of
$H_1(\bar{B})$. As explained in the above remark, this does not depend on the choice of $\delta$,
and it coincides with the linking number $\ell k_f(P, x_0)$ when $x_0 \in \mathrm{Fix}(f) \setminus C_n$.

\begin{remark}
Note that the linking number could equivalently be defined directly, without
the isotopy $f_t$, as the sum over $P$ of the coefficients of
$[f^{-1}(\delta) \cdot \delta^{-1}] \in H_1(B^3 \setminus C_n)$. We nevertheless
use the definition via the motion of the isotopy $f_t$, since it explains the geometric meaning
of $\ell k_f(P, x_0)$. That is, it measures how the trajectory of $x_0$ winds around the
trajectories of the circles of $P$ during $f_t$, in analogy with the classical
linking numbers of fixed points with braids in the disk. The same holds in the classical case of the disk with fixed boundary. However, if the
boundary is not fixed, then closing the loop requires a path along the circle
$\partial \mathbb{D}^2$, which is only determined by the isotopy. In our setting
this issue does not arise, since the boundary $\partial B^3$ is $S^2$, which is simply connected.
\end{remark}

\subsection{Abelianized generalized Lefschetz number}\label{abelian}

Let $f \in \mathrm{Homeo}(B^3, C_n)$. The fundamental group of $B^3\setminus C_n$ is the free group $F_n$ on $n$ generators, which we denote by $x_1, \dots, x_n$. This follows from the fact that the unlink complement $B^3\setminus C_n \subset B^3$ deformation retracts 
onto a wedge of $n$ circles and $n$ copies of the $2$-sphere. Let $f_*$ be the induced automorphism of $H_1(B^3\setminus C_n)$, and let
$$\operatorname{Coker}(f_* - \mathrm{id})
=
H_1(B^3\setminus C_n) \big/ \operatorname{Im}(f_* - \mathrm{id}).$$

We recall that $b_{f,C_n}$ induces a permutation on the $n$ circles in $C_n$. Let
$\mu$ be the permutation corresponding to the braid $b_{f,C_n}$, that is,
$f(S^1_i) = S^1_{\mu(i)}$. Suppose $\mu$ is decomposed into $m$ cycles
$\mu_1, \dots, \mu_m$. For $i = 1, \dots, n$, let $a_i \in H_1(B^3 \setminus C_n)$
be the image of $x_i$ under the abelianization homomorphism
\[
\mathrm{Ab} : \pi_1(B^3 \setminus C_n) \to H_1(B^3 \setminus C_n).
\]
The elements $a_1, \dots, a_n$ form a basis of the free abelian group
$H_1(B^3 \setminus C_n)$. The automorphism $f_*$ of $H_1(B^3 \setminus C_n)$ is the
abelianization of the automorphism $f_\pi$ of $\pi_1(B^3 \setminus C_n)$, that is,
$\mathrm{Ab} \circ f_\pi = f_* \circ \mathrm{Ab}$. Since $f$ maps $S^1_i$ to
$S^1_{\mu(i)}$ preserving orientations, we have $f_*(a_i) = a_{\mu(i)}$. Hence
$\operatorname{Coker}(f_* - \mathrm{id})$ is obtained from
$H_1(B^3 \setminus C_n)$ by identifying the $a_i$ within each cycle of $\mu$, which
is free abelian with basis $t_1, \dots, t_m$, where $t_j = [a_i]$ for any
$i \in \mu_j$. This explains the notation $\Lambda_\mu$ below. Let
\[
\Lambda = \mathbb{Z} H_1(B^3 \setminus C_n) = \mathbb{Z}[a_1^{\pm 1}, \dots, a_n^{\pm 1}],
\]
where $\Lambda$ is the ring of Subsection~\ref{matrixR}, and let
\[
\Lambda_\mu = \mathbb{Z}\operatorname{Coker}(f_* - \mathrm{id})
= \mathbb{Z}[t_1^{\pm 1}, \dots, t_m^{\pm 1}].
\]
The quotient map $H_1(B^3 \setminus C_n) \to \operatorname{Coker}(f_* - \mathrm{id})$
induces the ring homomorphism
\[
\pi_\mu : \Lambda \to \Lambda_\mu, \qquad \pi_\mu(a_i) = t_j \ \text{ for } i \in \mu_j.
\]
The abelianization from $F_n$ to $H_1(B^3 \setminus C_n)$ extends linearly
to $\mathrm{Ab} : \mathbb{Z}F_n \to \Lambda$. Moreover, if $\lambda_1, \lambda_2$ are
two representatives of the same Reidemeister class, that is,
$\lambda_2 = f_\pi(\lambda)\, \lambda_1\, \lambda^{-1}$ for some $\lambda \in \pi$,
then the element $\mathrm{Ab}(\lambda_2) - \mathrm{Ab}(\lambda_1) =
(f_* - \mathrm{id})(\mathrm{Ab}(\lambda))$ is trivial in the cokernel. Hence
$\mathrm{Ab}$ induces a well-defined map
$\mathrm{ab} : \mathcal{R}(f_\pi) \to \operatorname{Coker}(f_* - \mathrm{id})$ and
we denote by
$$
\overline{\mathrm{Ab}} : \mathbb{Z}\mathcal{R}(f_\pi) \to \Lambda_\mu
$$
its linear extension, so that $\overline{\mathrm{Ab}}([\alpha]) =
(\pi_\mu \circ \mathrm{Ab})(\alpha)$ for $\alpha \in \mathbb{Z}F_n$. 

Let $\bar{B}$ be the compactification of $B^3\setminus C_n$, and let
$\bar{f} : \bar{B} \to \bar{B}$ be the blow-up of $f$ at $C_n$.
Then $\bar{B}$ is a finite cell complex, and its fundamental group is also identified with $F_n$. We are now ready to present the following result about the abelianized generalized Lefschetz number of $\bar{f}$. The following lemma also applies to the $2$-dimensional case of homeomorphisms on the punctured 2-disk, see Lemma 4.7 in \cite{matsuoka}.  First, we need to fix some notation.
Suppose that the set $C_n$ is decomposed into $m$ periodic orbits $P_j$. That is, let $P_j\subseteq C_n$ be the union of the components $S^1_i$ with $i\in\mu_j$. For 
\[
I = (i_1, \dots, i_m) \in \mathbb{Z}^m,
\]
let $\operatorname{Fix}_I(\bar{f})$ be the set of fixed points $x$ of $\bar{f}$ such
that $\big(\ell k_{\bar{f}}(P_1, x), \dots, \ell k_{\bar{f}}(P_m, x)\big) = I$. We note that in the proof of the following lemma, we write the group operation of $\operatorname{Coker}(f_* - \mathrm{id})$
additively, so that its elements are of the form $\sum_j k_j t_j$, while in the group
ring $\Lambda_\mu$ the same element is written multiplicatively, as the monomial
$t_1^{k_1} \cdots t_m^{k_m}$.

\begin{lemma}\label{lemma}
\[
\overline{\mathrm{Ab}}\bigl(\mathcal{L}(\bar{f})\bigr)
=
\sum_{I\in \mathbb{Z}^m}
\operatorname{ind}\bigl(\operatorname{Fix}_I(\bar{f})\bigr)
\, t_1^{i_1} \cdots t_m^{i_m}
\in \Lambda_\mu.
\]

\end{lemma}

\begin{proof}
Let $F$ be a Nielsen class of $\bar{f}$ and let $x \in F$. We take the base point
of Section~\ref{Lefschetz} to be the point $z \in \partial B^3$. Since $\bar{f}$
fixes $z$, we take $w$, the path from $z$ to $\bar{f}(z)$, to be the constant path. Let $\ell$ be a path in $\bar{B}$
from $z$ to $x$. Then the coordinate $R(F) = R(x)$ is represented by the element
of $\pi_1(\bar{B})$ given by the loop $\alpha = (\bar{f} \circ \ell)\, \ell^{-1}$ based at $z$, and
$\mathrm{Ab}(R(x))$ is the homology class $[\alpha] \in H_1(\bar{B})$.

We want to compare $\alpha$ with the loop $\beta = \bar{f}^{-1}(\ell^{-1}) \cdot \ell$
at $x$. The loop $\beta$ is the loop used to define the linking numbers of $x$, with
$\delta = \ell^{-1}$, as defined in Subsection \ref{linking}, and more precisely its homology class is
$[\beta] = [\hat{\gamma}_x] = \sum_{i=1}^n l_i m_i .$
The loop $\bar{f}(\beta) = \ell^{-1} \cdot (\bar{f} \circ \ell)$ is the loop
$\alpha$ with a different base-point and so they have the same homology class. That is,
\[
\mathrm{Ab}(R(x)) = [\alpha] = [\bar{f}(\beta)] = f_*[\beta] = f_*[\hat{\gamma}_x].
\]
The elements $f_*[\hat{\gamma}_x]$ and $[\hat{\gamma}_x]$ have the same image in
$\operatorname{Coker}(f_* - \mathrm{id})$ because, for any $v \in H_1(\bar{B})$, we have $f_*(v) - v = (f_* - \mathrm{id})(v)$. Therefore $\mathrm{ab}(R(F))$, the image of $\mathrm{Ab}(R(x))$ in
$\operatorname{Coker}(f_* - \mathrm{id})$, is equal to the image of
$[\hat{\gamma}_x] = \sum_{i=1}^n l_i m_i$ under the map $\mathrm{ab}$. By definition we have $m_i = a_i$, and also we know that the image of $a_i$ in
the cokernel is $t_j$ for $i \in \mu_j$. Thus, we obtain, writing the cokernel
additively,
\[
\mathrm{ab}\big(R(F)\big) = \sum_{j=1}^m \Big(\sum_{i \in \mu_j} l_i\Big) t_j .
\]
By definition $P_j$ is the subset of $C_n$ that consists of the components $S^1_i$ with $i \in \mu_j$. Then, by definition of the linking number, the sum $\sum_{i \in \mu_j} l_i$ is actually the linking number
$\ell k_{\bar{f}}(P_j, x)$. Therefore,
\[
\mathrm{ab}\big(R(F)\big) = \sum_{j=1}^m \ell k_{\bar{f}}(P_j, x)\, t_j .
\]
To simplify the notation we let $\ell_j(F) := \ell k_{\bar{f}}(P_j, x)$. We know that all points of $F$ have the same
Reidemeister class. Thus, $\ell_j(F)$ does not depend on the choice of $x \in F$. In the
group ring $\Lambda_\mu$, written multiplicatively, this means
$\overline{\mathrm{Ab}}(R(F)) = t_1^{\ell_1(F)} \cdots t_m^{\ell_m(F)}$. Therefore,
by the definition of the generalized Lefschetz number,
\[
\overline{\mathrm{Ab}}\bigl(\mathcal{L}(\bar{f})\bigr)
= \sum_{F \in \mathrm{NC}(\bar{f})} \operatorname{ind}(F)\,
\overline{\mathrm{Ab}}\bigl(R(F)\bigr)
= \sum_{F \in \mathrm{NC}(\bar{f})} \operatorname{ind}(F)\,
t_1^{\ell_1(F)} \cdots t_m^{\ell_m(F)}.
\]

Finally, for each $I = (i_1, \dots, i_m) \in \mathbb{Z}^m$, the Nielsen classes $F$ with linking numbers $(\ell_1(F), \dots, \ell_m(F)) = I$ all contribute the same monomial
$t_1^{i_1} \cdots t_m^{i_m}$.
The sum of their indices is
$\operatorname{ind}(\operatorname{Fix}_I(\bar{f}))$ and finally we obtain
\[
\overline{\mathrm{Ab}}\bigl(\mathcal{L}(\bar{f})\bigr)
=
\sum_{I\in \mathbb{Z}^m}
\operatorname{ind}\bigl(\operatorname{Fix}_I(\bar{f})\bigr)
\, t_1^{i_1} \cdots t_m^{i_m}.
\]

\end{proof}

\begin{remark}
    From Lemma \ref{lemma} it is important to note that each monomial $t_1^{i_1} \cdots t_m^{i_m}$ in the expression of $\overline{\mathrm{Ab}}\bigl(\mathcal{L}(\bar{f})\bigr)$ corresponds to at least one fixed point $x$ with non-zero coefficient, which can be considered as the representative of one Nielsen class and that each exponent corresponds to the linking number of the fixed point $x$ and a subset of $C_n$. To be more precise, for a loop braid $b_{f,C_n}$ we consider the induced permutation that corresponds to the permutation of the $n$ circles in $C_n$. For example, let $b_{f,C_4}= \sigma_1\rho_3\in \LB_4$. The induced permutation $\mu$ on $C_4=\{S^1_1,S^1_2, S^1_3, S^1_4 \}$  corresponds to $(12)(34)\in S_4$ and decomposes into 2 disjoint cycles $\mu_1=(12)$ and $\mu_2=(34)$. Then, $\pi_\mu(a_1) = t_1, \pi_\mu(a_2) = t_1, \pi_\mu(a_3) = t_2$ and $\pi_\mu(a_4) = t_2$. Thus, the exponent $i_1$ corresponds to the linking number of $x$ and $P_1:=\{S^1_1, S^1_2\}$ and the exponent $i_2$ corresponds to the linking number of $x$ and $P_2:=\{S^1_3, S^1_4\}$.
\end{remark}

\subsection{Action on the universal covering space}\label{universal}
In Section \ref{Lefschetz}, we recalled that 
$$
\mathcal{L}(f) = \sum_q (-1)^q \, \mathrm{tr}\, \widetilde{f}_q \in \mathbb{Z}\mathcal{R}(f_\pi),
$$
where $\tilde{X}$ denotes the universal covering space of $X$ and that $\mathcal{L}(f)$ is a homotopy invariant. Let $X=B^3\setminus C_n$. We know that the complement of the $n$-component trivial link in $B^3$ deformation retracts onto  a wedge of $n$ circles and $n$ copies of
the 2-sphere. That is, $B^3\setminus C_n$ is homotopy equivalent to $\bigvee^n_i S^1_i \vee \bigvee^n_i S^2_i$. Since $\mathcal{L}$ is a homotopy invariant, we may compute it using the usual cell structure of
this wedge, consisting of a single $0$-cell $x_0$, $n$ $1$-cells $e_1, \dots, e_n$, and $n$
$2$-cells $b_1, \dots, b_n$, where $b_i$ corresponds to a $2$-sphere in $X$ enclosing $S^1_i$ and
no other component of $C_n$. Therefore, $$
\mathcal{L}(f) = [1] - \mathrm{tr}\, \widetilde{f}_1 + \mathrm{tr}\, \widetilde{f}_2 \in \mathbb{Z}\mathcal{R}(f_\pi),
$$
where $\widetilde{f}_1: C_1(\widetilde{X}) \to C_1(\widetilde{X})$ and $\widetilde{f}_2: C_2(\widetilde{X}) \to C_2(\widetilde{X})$.

We now need to determine the induced action of $f$ on the universal covering space. That is, the maps $\tilde{f}_1$ and $\tilde{f}_2$. To determine the universal covering space of $X$, that is $\tilde{X}$, we start with the universal cover of $\bigvee^n_i S^1_i$, which is the Cayley graph of the free group $F_n$. Let $x_0$ be the point of identification 
in the wedge.
To each lift $\tilde{x_0} \in \widetilde{X}$ of $x_0$, we attach one copy 
of the universal cover of $\bigvee^n_i S^2_i$. Since, $\bigvee^n_i S^2_i$ is simply connected we obtain that $\tilde{X}$ is actually the Cayley graph of the free group $F_n$, where at each vertex we attach one copy of $\bigvee^n_i S^2_i$, see Figure \ref{FigureA}.

\begin{figure}
  \centering
\begin{tikzpicture}[
    scale=0.8,
    line width=0.6pt,
        every node/.style={font=\small},
    sphere/.style={insert path={
        (0,0) -- (-0.138,0.683) arc[start angle=11.4, end angle=253.6, radius=0.42] -- cycle}},
    equator front/.style={insert path={
        (-0.97,0.6) arc[start angle=180, end angle=360, x radius=0.4245, y radius=0.12]}},
    equator back/.style={insert path={
        (-0.97,0.6) arc[start angle=180, end angle=0, x radius=0.4245, y radius=0.12]}},
    petals/.pic={
      \foreach \m in {1,-1} {
        \begin{scope}[xscale=\m]
          \draw[sphere];
          \draw[equator front];
          \draw[densely dashed, equator back];
        \end{scope}
      }
      \fill (0,0) circle (1.6pt);
    },
    small petals/.style={line width=0.35pt},
    midarrow/.style={postaction={decorate}, decoration={markings,
        mark=at position 0.55 with {\arrow{>}}}},
    cpath/.style={violet!80!black, line width=1pt}
  ]
  \draw (-6.6,0) -- (6.6,0);
  \draw (0,-5) -- (0,5);
  \foreach \y in {4,-4} \draw (-1.6,\y) -- (1.6,\y);
  \foreach \x in {4,-4} \draw (\x,-3.4) -- (\x,3.4);
  \foreach \x in {4,-4} \foreach \y in {2.35,-2.35} {
    \draw (\x-2.35,\y) -- (\x+2.35,\y);
    \foreach \s in {-1.6,1.6} \draw (\x+\s,\y-0.55) -- (\x+\s,\y+0.55);
  }
  \pic[scale=1.25] at (0,0) {petals};
  \foreach \p in {(4,0),(-4,0),(0,4),(0,-4)} \pic[scale=0.95] at \p {petals};
  \foreach \x in {4,-4} \foreach \y in {2.35,-2.35} {
    \pic[scale=0.95] at (\x,\y) {petals};
    \foreach \s in {-1.6,1.6} \pic[scale=0.42, small petals] at (\x+\s,\y) {petals};
  }
  \draw[cpath, midarrow] (0,0) -- (4,0);
  \draw[cpath, midarrow] (4,0) -- (4,2.35);
  \draw[cpath, midarrow] (4,2.35) -- (2.4,2.35);
  \fill (2.4,2.35) circle (1.8pt);
  \node at (-1.75,1.35) {$\tilde{b}_1$};
  \node at (1.75,1.35) {$\tilde{b}_2$};
  \node[violet!80!black, below] at (2,-0.05) {$\tilde{x}_1$};
  \node[below right] at (4,0) {$x_1$};
  \node[below left] at (0,4) {$x_2$};
  \node at (2.4,0.8) {$x_1\tilde{b}_1$};
  \node at (5.8,0.8) {$x_1\tilde{b}_2$};
  \node[violet!80!black, left] at (4.1,1.7) {$x_1\tilde{x}_2$};
  \node[below right] at (4,2.35) {$x_1x_2$};
    \node[above left] at (2.0,2.35) {$x_1x_2x_1^{-1}$};
  \node[violet!80!black] (lab) at (1.12,1.95) {$\overline{x_1x_2x_1^{-1}\tilde{x}_1}$};
    \draw[densely dashed, ->, line width=0.5pt] ([xshift=-5pt, yshift=5.5pt]lab.south east) .. controls (2.3,1.45) and (2.75,1.5) .. (2.75,2.27);
\end{tikzpicture} 
\caption{Universal covering space of $X$ for $n=2$.}\label{FigureA}
\end{figure}
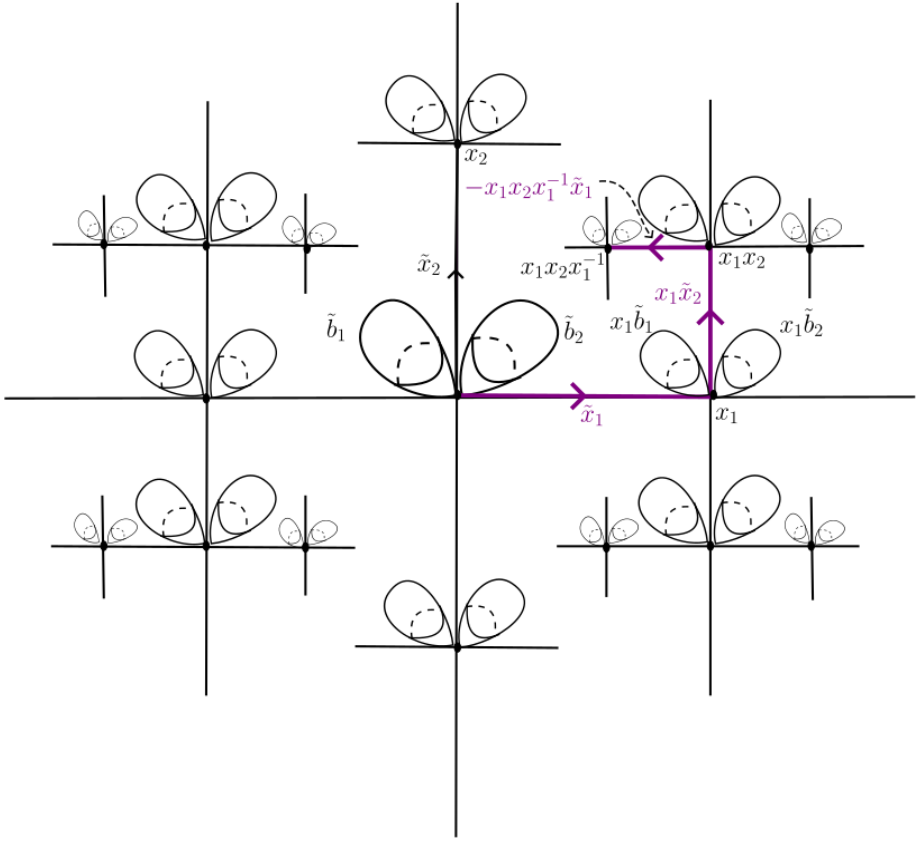

Let $b \in \LB_n$. As explained
before Definition \ref{automorphism}, $b$ induces an automorphism of
$\pi_1(X) \cong \langle x_1, \dots, x_n \rangle$, where we identify $x_1,\dots,x_n$ with the $1$-cells $e_1,\dots,e_n$ of $X$. Below we recall the action of the generators $\sigma_i$ and $\rho_i$ of $\LB_n$ on the 1-cells.

\begin{equation*}
\begin{minipage}{0.45\textwidth}
\[
\sigma_i:
\begin{cases}
x_i \mapsto x_i x_{i+1} x_i^{-1} \\
x_{i+1} \mapsto x_i \\
x_j \mapsto x_j, \quad j \neq i, i+1
\end{cases}
\]
\end{minipage}
\hfill
\begin{minipage}{0.45\textwidth}
\[
\rho_i:
\begin{cases}
x_i \mapsto x_{i+1} \\
x_{i+1} \mapsto x_i \\
x_j \mapsto x_j, \quad j \neq i, i+1.
\end{cases}
\]
\end{minipage}
\end{equation*}

We want to know how $\sigma_i$ and $\rho_i$ act on the 1-cells of the universal covering space. 
Let $\widetilde{x}_1, \dots, \widetilde{x}_n$ correspond to lifts of the $1$-cells $x_1, \dots, x_n$ with initial point $\widetilde{x}_0$. From the action of $\sigma_i$ on the $\{x_1,\dots,x_n\}$ it follows that the lift of the homomorphism $\sigma_i$ keeps $\tilde{x}_j$ fixed for $j\neq i, i+1$, transforms $\tilde{x}_{i+1}$ into $\tilde{x}_{i}$ and stretches $\tilde{x}_{i}$ into the path $\tilde{x_i}(x_i\tilde{x}_{i+1})(\overline{x_ix_{i+1}x_i^{-1}\tilde{x}_{i}})$, see Figure \ref{FigureA}. Note that we denote by $\overline{x_ix_{i+1}x_i^{-1}\tilde{x}_{i}}$ the reverse of the path $x_ix_{i+1}x_i^{-1}\tilde{x}_{i}$. Thus, the matrix over $\mathbb{Z}\pi$ representing the lift of $\sigma_i$ to $\tilde{\sigma_i}:C_1(\tilde{X})\to C_1(\tilde{X})$ with respect to the basis $\{\widetilde{x}_1, \dots, \widetilde{x}_n\}$ is the following:

\[
A_{1,\sigma_i}=I_{i-1} \oplus 
\begin{pmatrix}
1-x_ix_{i+1}x_i^{-1} & x_i \\
1 & 0
\end{pmatrix}
\oplus I_{n-i-1}.
\] 
From the action of $\rho_i$ on the $\{x_1,\dots,x_n\}$ it follows that the lift of the homomorphism $\rho_i$ keeps $\tilde{x}_j$ fixed for $j\neq i, i+1$, transforms $\tilde{x}_{i}$ into $\tilde{x}_{i+1}$ and $\tilde{x}_{i+1}$ into $\tilde{x}_{i}$. Thus, the matrix over $\mathbb{Z}\pi$ representing the lift of $\rho_i$ to $\tilde{\rho_i}:C_1(\tilde{X})\to C_1(\tilde{X})$ with respect to the basis $\{\widetilde{x}_1, \dots, \widetilde{x}_n\}$ is the following:

\[
A_{1,\rho_i}=I_{i-1} \oplus 
\begin{pmatrix}
0 & 1\\
1 & 0
\end{pmatrix}
\oplus I_{n-i-1}.
\] 

\begin{remark}
In the case of the punctured 2-disk in order to deduce the action of the homeomorphism on the $1$-cells of the universal covering space one can use the Fox derivatives, see \cite{huangJiang1989}. In our case, we could also have used the Fox derivatives by considering a different basis of $C_1(\tilde{X})$. 
\end{remark}

\begin{figure}[H]
	\centering
	\includegraphics[width=.8\textwidth]{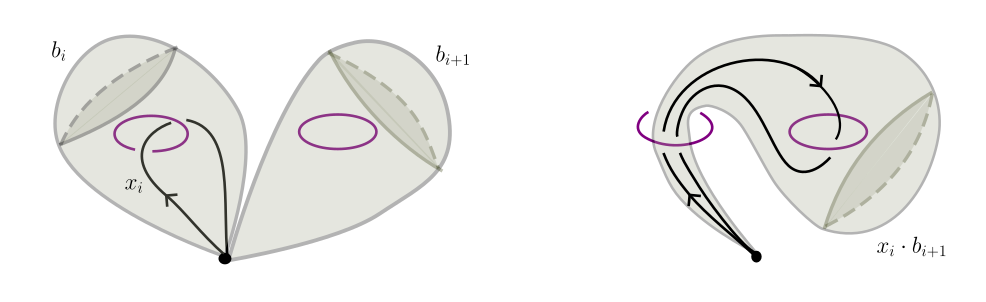}
	\caption{On the left hand side we see the $2$-cells $b_i$ and $b_{i+1}$ and on the right hand side we see the image of $b_i$ under the action of $\sigma_i$.}\label{FigureB}
\end{figure}

It remains to determine the action of $\sigma_i$ and $\rho_i$ on the 2-cells of the universal covering space. 
Let $\widetilde{b}_1, \dots, \widetilde{b}_n$ correspond to lifts of the $2$-cells $b_1, \dots, b_n$ with initial point $\widetilde{x}_0$. From the action of $\sigma_i$ on the $\{x_1,\dots,x_n\}$ it follows that $\sigma_i$ sends $b_i$ to $x_i\cdot b_{i+1}$, where $\cdot$ denotes the action of $\pi_1$ on $\pi_2$, see Figure \ref{FigureB}, it sends $b_{i+1}$ to $b_i+b_{i+1}-x_i\cdot b_{i+1}$, see Figure \ref{FigureC}, and keeps $b_j$ fixed for $j\neq i, i+1$. It follows that the lift of the homomorphism $\sigma_i$ keeps $\tilde{b}_j$ fixed for $j\neq i, i+1$, transforms $\tilde{b}_{i}$ into $x_i\tilde{b}_{i+1}$ and $\tilde{b}_{i+1}$ into $\tilde{b}_{i}+\tilde{b}_{i+1}-x_i\tilde{b}_{i+1}$. Thus, the matrix over $\mathbb{Z}\pi$ representing the lift of $\sigma_i$ to $\tilde{\sigma_i}:C_2(\tilde{X})\to C_2(\tilde{X})$ with respect to the basis $\{\widetilde{b}_1, \dots, \widetilde{b}_n\}$ is the following:

\[
A_{2,\sigma_i}=I_{i-1} \oplus 
\begin{pmatrix}
0 & x_i\\
1 & 1- x_i
\end{pmatrix}
\oplus I_{n-i-1}.
\] 
From the action of $\rho_i$ on the $\{x_1,\dots,x_n\}$ it follows that $\rho_i$ sends $b_i$ to $b_{i+1}$, $b_{i+1}$ to $b_i$ and keeps $b_j$ fixed for $j\neq i, i+1$. Thus, the lift of the homomorphism $\rho_i$ keeps $\tilde{b}_j$ fixed for $j\neq i, i+1$, transforms $\tilde{b}_{i}$ into $\tilde{b}_{i+1}$ and $\tilde{b}_{i+1}$ into $\tilde{b}_{i}$. Therefore, the matrix over $\mathbb{Z}\pi$ representing the lift of $\rho_i$ to $\tilde{\rho_i}:C_2(\tilde{X})\to C_2(\tilde{X})$ with respect to the basis $\{\widetilde{b}_1, \dots, \widetilde{b}_n\}$ is the following:

\[
A_{2,\rho_i}=I_{i-1} \oplus 
\begin{pmatrix}
0 & 1\\
1 & 0
\end{pmatrix}
\oplus I_{n-i-1}.
\] 

Finally, we define the map 
\begin{equation}\label{matrix1}
A_1:\LB_n\to GL_n(\mathbb{Z}F_n)\ \text{by}
\end{equation}
$$A_1(\sigma_i)=A_{1,\sigma_i},\ A_1(\rho_i)=A_{1,\rho_i}\ \text{and}\ A_1(bb')=A_1(b)^{\nu(b')}A_1(b').$$
Note that the matrix $A_1$ over $\mathbb{Z}\pi$ represents the $f_\pi$-homomorphism $\widetilde{f}_1$ with respect to the basis of $C_1(\widetilde{X})$ over $\mathbb{Z}\pi$ consisting of the 1-cells $\tilde{x}_1, \dots, \tilde{x}_n$. Similarly, we define the map 
\begin{equation}\label{matrix2}
A_2:\LB_n\to GL_n(\mathbb{Z}F_n)\ \text{by}
\end{equation}
$$A_2(\sigma_i)=A_{2,\sigma_i},\ A_2(\rho_i)=A_{2,\rho_i}\ \text{and}\ A_2(bb')=A_2(b)^{\nu(b')}A_2(b').$$
Note that the matrix $A_2$ over $\mathbb{Z}\pi$ represents the $f_\pi$-homomorphism $\widetilde{f}_2$ with respect to the basis of $C_2(\widetilde{X})$ over $\mathbb{Z}\pi$ consisting of the 2-cells $\tilde{b}_1, \dots, \tilde{b}_n$. 

\begin{figure}[H]
	\centering
	\includegraphics[width=.4\textwidth]{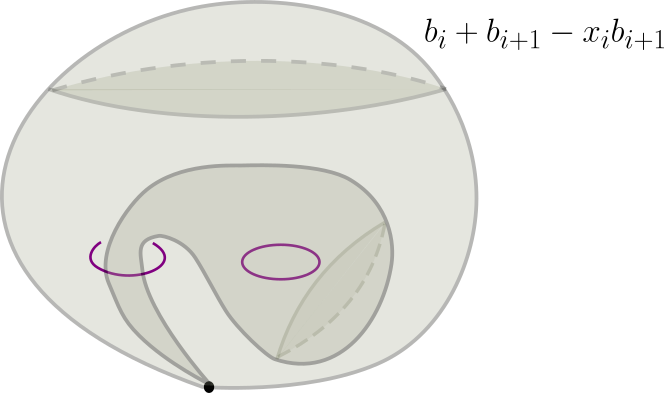}
	\caption{The image of $b_{i+1}$ under the action of $\sigma_i$.}\label{FigureC}
\end{figure}

We note that the matrices defined above are matrices by rows. That is, the $i^{th}$ row of the matrix of $\tilde{f}_q$ gives
$\tilde{f}_q(\tilde{e}_i)$, which gives the rule
$A_q(bb') = A_q(b)^{\nu(b')} A_q(b')$, where $bb'$ means $b$ followed by $b'$. With the usual column convention all matrices would be
transposed, which does not affect the traces. 

\subsection{Main theorem}
We are now ready to present and prove our main theorem, but let us first introduce a homomorphism that we will need for the theorem.
Let $\mu$ be a permutation of $\{1, \dots, n\}$, decomposed as a product
of $m$ disjoint cycles $\mu_1, \dots, \mu_m$. We consider $m$ variables
$t_1, \dots, t_m$ and let
$\Lambda_\mu = \mathbb{Z}[t_1^{\pm 1}, \dots, t_m^{\pm 1}]$. For $\Lambda = \mathbb{Z}[a_1^{\pm1}, \dots, a_n^{\pm1}]$, consider the surjective homomorphism
\[
\pi_\mu : \Lambda \to \Lambda_\mu, \qquad \pi_\mu(a_i) = t_j,
\]
where $j$ is determined by the condition $i \in \mu_j$. For example, the permutation $\mu$ of $\{1,2,3,4\}$ that corresponds to $(12)$ and $(34)$ decomposes into 2 disjoint cycles $\mu_1=(12)$ and $\mu_2=(34)$. Then, $\pi_\mu(a_1) = t_1, \pi_\mu(a_2) = t_1, \pi_\mu(a_3) = t_2$ and $\pi_\mu(a_4) = t_2$.

 Throughout this work, for a matrix $R(b)\in GL_n(\Lambda)$ we denote by $\operatorname{tr}(R(b)^{\pi_\mu})$ the trace obtained from the matrix $R(b)$ after having replaced each entry by its image under $\pi_\mu$. Similarly, we use the same notation for matrices $\bar{R}(b)\in GL_n(\Lambda)$. 

\begin{theorem}\label{thmmain}
Let $b = b_{f,C_n}$. Then
\[
1+\operatorname{tr}(S(b)^{\pi_\mu}) = \sum_{I \in \mathbb{Z}^m} \operatorname{ind}\big(\operatorname{Fix}_I(\bar{f})\big) \, t_1^{i_1} \cdots t_m^{i_m} \in \mathbb{Z}[t_1^{\pm 1}, \dots, t_m^{\pm 1}],
\]
where $S(b)^{\pi_\mu}=\bar{R}(b)^{\pi_\mu}-R(b)^{\pi_\mu}$.
\end{theorem}

\begin{proof}
 
Let $b=b_{f,C_n}\in \LB_n$. From Lemma \ref{lemma} we have 
$$\overline{\mathrm{Ab}}(\mathcal{L}(\bar{f}))=\sum_{I \in \mathbb{Z}^m} \operatorname{ind}(\operatorname{Fix}_I(\bar{f})) \, 
t_1^{i_1} \cdots t_m^{i_m}.$$ 
The inclusion of $B^3 \setminus C_n$ into $\bar{B}$ is a homotopy equivalence, and
$\bar{f}$ coincides with $f$ on $B^3 \setminus C_n$. Hence, since the generalized Lefschetz
number is a homotopy invariant, we may compute $\mathcal{L}(\bar{f})$ using $f$ and the cell
structure of $B^3 \setminus C_n$ described in Subsection \ref{universal}.
Thus, from \eqref{trace} it follows that 
$$\overline{\mathrm{Ab}}(\mathcal{L}(\bar{f}))= \overline{\mathrm{Ab}}\big([1] - [\mathrm{tr}\, \widetilde{f}_1] + [\mathrm{tr}\, \widetilde{f}_2 ]\big).$$
By the defined matrices in \eqref{matrix1}, \eqref{matrix2} and from the fact that, $$\text{for}\ \alpha \in \mathbb{Z}F_n,\ \text{and}\ [\alpha]\in \mathbb{Z}\mathcal{R}(f_{\pi}), \
\overline{\mathrm{Ab}}([\alpha]) = (\pi_\mu \circ \mathrm{Ab})(\alpha),$$ we obtain
the equality
$$\overline{\mathrm{Ab}}\big([1] - [\mathrm{tr}\, \widetilde{f}_1] + [\mathrm{tr}\, \widetilde{f}_2] \big)= (\pi_\mu \circ \mathrm{Ab})\big(1 - \mathrm{tr}\,A_1(b) + \mathrm{tr}\,A_2(b)\big).$$ After the abelianization of the entries of the matrices $A_1$ and $A_2$ the resulting matrices coincide with the matrices, defined in Subsection \ref{matrixR}, $R$ and $\bar{R}$ respectively. Thus,
$$\overline{\mathrm{Ab}}(\mathcal{L}(\bar{f}))= \pi_\mu\big(1 - \operatorname{tr} R(b)+\operatorname{tr} \bar{R}(b)\big)$$
Let $S(b):=\bar{R}(b)-R(b)$. Then,
$$\overline{\mathrm{Ab}}(\mathcal{L}(\bar{f}))=\pi_\mu\big(1 + \operatorname{tr}  S(b)\big).$$
We conclude that 
$$1+ \operatorname{tr}(S(b)^{\pi_\mu})=\sum_{I \in \mathbb{Z}^m} \operatorname{ind}(\operatorname{Fix}_I(\bar{f})) \, 
t_1^{i_1} \cdots t_m^{i_m}.$$
\end{proof}

Let $\mathbb{Z}[t^{\pm1}]$ be the ring of Laurent polynomials over $\mathbb{Z}$ in the variable $t$. We define the homomorphism
$$
\pi : \Lambda \to \mathbb{Z}[t^{\pm1}]
\ \text{by}\ \pi(a_i) = t\ \text{for each}\ i.$$
The image of $R(b)$ under this homomorphism $\pi$ coincides with the
\emph{unreduced Burau matrix}, and the homomorphism
\[
\pi \circ R : \LB_n \to GL_n(\mathbb{Z}[t^{\pm1}])
\]
is called the \emph{unreduced Burau representation} of the loop braid groups $\LB_n$. Moreover, the matrix $R(b)^\pi$ coincides with the image of
$R(b)^{\pi_\mu}$ under the homomorphism from
$\Lambda_\mu$ to $\mathbb{Z}[t^{\pm1}]$ sending each $t_j$ to $t$.
Similarly, $\bar{R}$ is related to the Burau representation. For each generator, transposing
$\bar{R}(\sigma_i)^\pi$ and reversing the order of the basis gives the unreduced Burau matrix of
$\sigma_{n-i}$, and similarly for $\rho_i$. 

\section{Examples and application}\label{Application}

We will conclude our work with two examples and with an application of our main result concerning the number of periodic points.

\subsection{Examples}
\begin{example}
Let $C_3=S^1_1\sqcup S^1_2\sqcup S^1_3$ and $f\in \Homeo(B^3, C_3)$ such that $b=b_{f,C_3}=\sigma_1\rho_2\sigma_1\in \LB_3$. The induced permutation of $C_3$ under $b$ corresponds to 
$\mu=(13)(2)\in S_3$. Let $\mu_1=(13)$ and $\mu_2=(2)$. It follows that $m = 2$, $\Lambda_\mu = \mathbb{Z}[t_1^{\pm 1}, t_2^{\pm 1}]$,
with $\pi_\mu(a_1) = t_1$, $\pi_\mu(a_2)=t_2$ and $\pi_\mu(a_3) = t_1$.

We will calculate $1+\operatorname{tr}(S(b)^{\pi_\mu})$. We have 
\begin{align*}
  1+\operatorname{tr}(S(b)^{\pi_\mu})
  &=1+\operatorname{tr}(\bar{R}(b)^{\pi_\mu}-R(b)^{\pi_\mu})\\
  &=1+\operatorname{tr}(\bar{R}(b)^{\pi_\mu})-\operatorname{tr}(R(b)^{\pi_\mu})\\
  &=1+\operatorname{tr}\begin{pmatrix}
0 & 0&a_2\\
0 & a_1 &1-a_2\\
1 & 1-a_1 &0
\end{pmatrix}^{\pi_\mu} - \operatorname{tr}\begin{pmatrix}
(1-a_3)(1-a_2) &(1-a_3)a_1 &a_2\\
1-a_2 & a_1 &0\\
1 & 0 &0
\end{pmatrix}^{\pi_\mu}\\
  &=1+\operatorname{tr}\begin{pmatrix}
0 & 0&t_2\\
0 & t_1 &1-t_2\\
1 & 1-t_1 &0
\end{pmatrix} - \operatorname{tr}\begin{pmatrix}
(1-t_1)(1-t_2) &(1-t_1)t_1 &t_2\\
1-t_2 & t_1 &0\\
1 & 0 &0
\end{pmatrix}\\
&= 1+ t_1-1+t_2+t_1-t_1t_2-t_1\\
&= t_1+t_2-t_1t_2
\end{align*}
From Theorem \ref{thmmain} we can conclude that 
$$\displaystyle\sum_{(i_1, i_2)\in\mathbb{Z}^2}\operatorname{ind}(\mathrm{Fix_{(i_1,i_2)}}(\bar{f}))t_1^{i_1}t_2^{i_2}=t_1+t_2-t_1t_2.$$
Hence $\operatorname{ind}(\mathrm{Fix}_{(1,0)}(\bar{f})) = \operatorname{ind}(\mathrm{Fix}_{(0,1)}(\bar{f})) = 1$ and
$\operatorname{ind}(\mathrm{Fix}_{(1,1)}(\bar{f})) = -1$. Thus, $\bar{f}$ has at least three fixed points, lying in three different Nielsen classes.
Let $P_1 = \{S^1_1, S^1_3\}$ and $P_2 = \{S^1_2\}$. There is a fixed point with linking numbers
$1$ with $P_1$ and $0$ with $P_2$, one with linking numbers $0$ with $P_1$ and $1$ with $P_2$,
and one with linking number $1$ with both $P_1$ and $P_2$.
\end{example}

\begin{example}
Let $C_4 = S^1_1 \sqcup S^1_2 \sqcup S^1_3 \sqcup S^1_4$ and let $f \in \Homeo(B^3, C_4)$ with
$b = b_{f, C_4} = \sigma_1 \sigma_2^2 \rho_3 \in \LB_4$. The induced permutation is
$\mu = (1\,2)(3\,4)$, with cycles $\mu_1 = (1\,2)$ and $\mu_2 = (3\,4)$, so that
$\pi_\mu(a_1) = \pi_\mu(a_2) = t_1$ and $\pi_\mu(a_3) = \pi_\mu(a_4) = t_2$. With direct calculations
we obtain
\[
R(b)^{\pi_\mu} =
\begin{pmatrix}
1 - t_1 & t_1 - t_1^2 + t_1^2 t_2 & 0 & t_1^2 - t_1^3 \\
1 & 0 & 0 & 0 \\
0 & 1 - t_2 & 0 & t_1 \\
0 & 0 & 1 & 0
\end{pmatrix},
\quad
\bar{R}(b)^{\pi_\mu} =
\begin{pmatrix}
0 & t_1 t_2 & 0 & t_1 t_2 - t_1^2 t_2 \\
1 & t_2 - t_1 t_2 & 0 & t_2 - 2 t_1 t_2 + t_1^2 t_2 \\
0 & 1 - t_2 & 0 & 1 - t_2 + t_1 t_2 \\
0 & 0 & 1 & 0
\end{pmatrix}.
\]
Hence $\operatorname{tr} \big(R(b)^{\pi_\mu}\big) = 1 - t_1$, $\operatorname{tr} \big(\bar{R}(b)^{\pi_\mu}\big) = t_2 - t_1 t_2$, and
by Theorem~\ref{thmmain} we have:
\[
\sum_{(i_1, i_2) \in \mathbb{Z}^2} \operatorname{ind}\big(\operatorname{Fix}_{(i_1, i_2)}(\bar{f})\big)\,
t_1^{i_1} t_2^{i_2} = t_1 + t_2 - t_1 t_2 .
\]
We note that $f$ has no fixed
points on $C_4$, because no component of $C_4$ is mapped to itself. Thus, by Remark \ref{extra}, $\operatorname{Fix}(\bar{f}) = \operatorname{Fix}(f)$.
Therefore $f$ has at least three fixed points in $B^3 \setminus C_4$, lying in three different
Nielsen classes. Let $P_1 = \{S^1_1, S^1_2\}$ and $P_2 = \{S^1_3, S^1_4\}$. There is a fixed point of $f$ with linking numbers $1$ with $P_1$ and $0$ with $P_2$, one fixed point of $f$ with linking numbers $0$ with $P_1$ and $1$ with $P_2$ and one with linking numbers $1$ with both $P_1$ and $P_2$. We know also that both the first and the second classes have index $+1$ while the last class has index $-1$.
\end{example}

\subsection{Application}

As an application of Theorem \ref{thmmain}, we obtain a lower bound for the number of
periodic points of $f$. Let $b = b_{f,C_n}$ and let $\mu$ be its permutation decomposed into $m$ cycles
$\mu_1, \dots, \mu_m$. Let $P_j \subseteq C_n$ be the orbit of components  that corresponds
to $\mu_j$. The permutation $\mu$ corresponds to the permutation of the $n$ circles $S^1_i$ in $C_n$. Thus, we can naturally associate a period to each circle $S_i^1$. That is, the length of the cycle $\mu_j$ containing $i$. 

 The map $f^p$ is an element in $\Homeo(B^3, C_n)$. If $f_t$ is an isotopy from $\mathrm{id}$ to $f$, then for $f^p$ we use the isotopy $f^{p}_t$
obtained by concatenating $p$ copies of $f_t$. That is, the $r^{th}$ copy is $f_t \circ f^{r-1}$, which goes
from $f^{r-1}$ to $f^r$, for $r = 1, \dots, p$.
The motion of $C_n$ under $f^{p}_t$ is $b_{f^p, C_n} = b^p$ and the trajectory
of a point $x$ is the concatenation of the trajectories of $x, f(x), \dots, f^{p-1}(x)$ under
$f_t$. Moreover, the blow-up of $f^p$ is $\bar{f}^p$ and since each $P_j$ is
$f$-invariant, it is also $f^p$-invariant. Thus, $\ell k_{\bar{f}^p}(P_j, x)$ is defined for
$x \in \operatorname{Fix}(\bar{f}^p)$. We recall from Subsection \ref{linking} that, for $x \in \operatorname{Fix}(\bar{f}^k)$ and an
$f$-invariant subset $P \subseteq C_n$, the linking number $\ell k_{\bar{f}^k}(P, x)$ is the sum,
over the components of $P$, of the coefficients of the class
$[\bar{f}^{-k}(\delta) \cdot \delta^{-1}] \in H_1(\bar{B})$, where $\delta$ is a path from $x$ to a
point $z \in \partial B^3$.

Before stating and proving the theorem about the lower bound of the number of periodic points of $f$ we need to make some essential remarks that will be used afterwards.

 We note that for every $q \geq 1$, we have 
 \begin{equation}\label{iterate-matrices}
 R(b^q)^{\pi_\mu} = \big(R(b)^{\pi_\mu}\big)^q \ \text{and} \
\bar{R}(b^q)^{\pi_\mu} = \big(\bar{R}(b)^{\pi_\mu}\big)^q.
 \end{equation}

Since $\nu(b)(a_i) = a_{\mu(i)}$ and $i$,
$\mu(i)$ lie in the same cycle of $\mu$, we have $\pi_\mu \circ \nu(b) = \pi_\mu$. Thus, by the product rule in \eqref{product}, that is, $R(b^q) = R(b^{q-1})^{\nu(b)} R(b)$ and by using the fact that $\pi_\mu \circ \nu(b) = \pi_\mu$, we get
$R(b^q)^{\pi_\mu} = R(b^{q-1})^{\pi_\mu} R(b)^{\pi_\mu}$, and the claim follows by induction. Similarly for $\bar{R}$.

We will use two properties of the defined linking numbers. Let $x \in \operatorname{Fix}(\bar{f}^k)$
and let $P \subseteq C_n$ be $f$-invariant. First, for every $q \geq 1$,
\begin{equation}\label{eq:lk-iterate}
\ell k_{\bar{f}^{kq}}(P, x) = q\, \ell k_{\bar{f}^k}(P, x).
\end{equation}
Indeed, the trajectory of $x$ under the concatenated isotopy $f^{kq}_t$ consists of $q$ copies of
its trajectory under $f^{k}_t$. During each copy the circles of $P$ perform the same motion, with their labels possibly permuted. Hence each trajectory of $x$ under $f^{k}_t$ contributes
$\ell k_{\bar{f}^k}(P, x)$. 
Second, the points $x$ and $\bar{f}(x)$ have the same linking numbers with $P$. That is,
\begin{equation}\label{eq:lk-orbit}
\ell k_{\bar{f}^k}(P, \bar{f}(x)) = \ell k_{\bar{f}^k}(P, x).
\end{equation}
Indeed, let $\delta$ be a path from $x$ to a point $z \in \partial B^3$, as above. Then
$\bar{f}(x)$ is fixed by $\bar{f}^k$, and $\bar{f}(\delta)$ is a path from $\bar{f}(x)$ to
$\bar{f}(z) = z$. With this path, the loop defining the linking numbers of $\bar{f}(x)$ is
$\bar{f}^{-k}(\bar{f}(\delta))\cdot\bar{f}(\delta)^{-1} = \bar{f}(L)$, where
$L = \bar{f}^{-k}(\delta)\cdot\delta^{-1}$. Thus, for $[L] = \sum_{i=1}^n l_i m_i$, we get
$[\bar{f}(L)] = f_*[L] = \sum_{i=1}^n l_i\, m_{\mu(i)}$. Therefore, the sum of the coefficients over $P$ is again $\sum_{S^1_i \in P} l_i$, since $\mu$ maps the indices of the components
of $P$ to themselves. In particular, all points
$x, \bar{f}(x), \dots, \bar{f}^{p-1}(x)$ of an orbit of period $p$ have the same linking numbers.

For $I = (i_1, \dots, i_m) \in \mathbb{Z}^m$, we denote by $\operatorname{Fix}_I(\bar{f}^p)$ the set of fixed points of $\bar{f}^p$
 with
$\big(\ell k_{\bar{f}^p}(P_1, x), \dots, \ell k_{\bar{f}^p}(P_m, x)\big) = I$.

\begin{proposition}\label{prop:iterate}
For every $p \geq 1$, the following equality holds:
\[
1 - \operatorname{tr}\big((R(b)^{\pi_\mu})^p\big) + \operatorname{tr}\big((\bar{R}(b)^{\pi_\mu})^p\big)
= \sum_{I \in \mathbb{Z}^m} \operatorname{ind}\big(\operatorname{Fix}_I(\bar{f}^p)\big)\,
t_1^{i_1} \cdots t_m^{i_m} \in \Lambda_\mu .
\]
\end{proposition}

\begin{proof}
For $b=b_{f,C_n}$ the permutation of $b^p$ is $\mu^p$ and is decomposed into the cycles $\nu_1, \dots, \nu_{m'}$. Each cycle $\nu_k$ of $\mu^p$ is contained in a cycle of $\mu$ and hence each cycle $\mu_j$ of $\mu$ is a union of cycles of $\mu^p$. Let $O_k \subseteq C_n$ be the union of the components $S^1_i$ with $i \in \nu_k$. Hence each $P_j$
is the disjoint union of the orbits $O_k$ with $\nu_k \subseteq \mu_j$, and by definition of the
linking number,
\begin{equation}\label{eq:lk-split}
\ell k_{\bar{f}^p}(P_j, x) = \sum_{\nu_k \subseteq \mu_j} \ell k_{\bar{f}^p}(O_k, x).
\end{equation}

Applying Theorem \ref{thmmain} to $f^p$ gives, in
$\Lambda_{\mu^p} = \mathbb{Z}[s_1^{\pm 1}, \dots, s_{m'}^{\pm 1}]$,
\[
1 - \operatorname{tr}\big(R(b^p)^{\pi_{\mu^p}}\big) + \operatorname{tr}\big(\bar{R}(b^p)^{\pi_{\mu^p}}\big)
= \sum_{J \in \mathbb{Z}^{m'}} \operatorname{ind}\big(\operatorname{Fix}_J(\bar{f}^p)\big)\,
s_1^{j_1} \cdots s_{m'}^{j_{m'}},
\]
where $\operatorname{Fix}_J(\bar{f}^p)$ is the set of fixed points of $\bar{f}^p$ whose linking
numbers with $O_1, \dots, O_{m'}$ are $j_1, \dots, j_{m'}$.

 Let $\kappa : \Lambda_{\mu^p} \to \Lambda_\mu$ be the ring homomorphism with
$\kappa(s_k) = t_j$ when $\nu_k \subseteq \mu_j$. Then $\kappa \circ \pi_{\mu^p} = \pi_\mu$, so
applying $\kappa$ to the left-hand side gives
$1 - \operatorname{tr}\big(R(b^p)^{\pi_\mu}\big) + \operatorname{tr}\big(\bar{R}(b^p)^{\pi_\mu}\big)$, which by \eqref{iterate-matrices} equals
$1 - \operatorname{tr}\big((R(b)^{\pi_\mu})^p\big) + \operatorname{tr}\big((\bar{R}(b)^{\pi_\mu})^p\big)$. Applying $\kappa$ on the right-hand side, $s_1^{j_1} \cdots s_{m'}^{j_{m'}}$ becomes
$t_1^{i_1} \cdots t_m^{i_m}$ with $i_j = \sum_{\nu_k \subseteq \mu_j} j_k$. By
\eqref{eq:lk-split}, this means that a fixed point in $\operatorname{Fix}_J(\bar{f}^p)$ has linking
numbers $i_1, \dots, i_m$ with $P_1, \dots, P_m$. After the substitution, different monomials $s_1^{j_1} \cdots s_{m'}^{j_{m'}}$ may become the
same monomial $t_1^{i_1} \cdots t_m^{i_m}$, whose coefficient is the sum of
their coefficients. Also, the exponents of the variables merged into
$t_j$ are added. Thus, since $P_j$ is the union of the orbits of $\mu^p$ contained in it, the resulting
exponent of $t_j$ is the linking number with $P_j$. Collecting the fixed points that give the same
monomial, and using the additivity of the index, we obtain the right-hand side of the proposition and this completes the proof.
\end{proof}

\begin{remark}
The passage to $\pi_\mu$ in the proof is necessary. By \eqref{product} we have $$R(b^p) = R(b)^{\nu(b)^{p-1}} \cdots R(b)^{\nu(b)}\, R(b),$$
where the $r^{th}$ factor is $R(b)$ with each $a_i$ replaced by $a_{\mu^r(i)}$. After a substitution
that identifies $a_i$ with $a_{\mu(i)}$ for all $i$, all these factors become equal and $R(b^p)$
becomes the $p^{th}$ power of a single matrix. This is the case for $\pi_\mu$, as explained in \eqref{iterate-matrices}, but in general not for $\pi_{\mu^p}$. The map $\pi_{\mu^p}$ sends $a_i$ and $a_{\mu(i)}$ to the same variable only
when $i$ and $\mu(i)$ lie in the same cycle of $\mu^p$, which is often not the case.
\end{remark}

\begin{remark}\label{one class}
Let $q \in \operatorname{Fix}(f^p) \cap C_n$, with $q \in S^1_i$. All fixed points of $\bar{f}^p$ in
the fiber $\{(q,u)\} \subset \mathbb{T}_i$ over $q$ belong to a single Nielsen class of $\bar{f}^p$. Indeed, the map $\bar{f}^p$ sends the
fiber of $\mathbb{T}_i$ over $q$ to itself by an orientation-preserving homeomorphism of the
circle, so all fixed points of $\bar{f}^p$ in this fiber belong to a single Nielsen class.
In the case of orientation-preserving homeomorphisms of the punctured disk, the boundary components of the blow-up are circles, and all
fixed points of $\bar{f}^p$ on one boundary circle lie in a single Nielsen class. In our setting the
boundary components are tori, and from the above observation we get one Nielsen class per
fixed point of $f^p$ on $C_n$ and not one class per circle $S^1_i$. That is, fixed points of $\bar{f}^p$ over different points of the same
circle $S^1_i$ may lie in different Nielsen classes. Indeed, between two fixed points $q, q'$ of $f^p$ on $S^1_i$, the derivative of $f^p$ may
make a full turn as one moves along the arc of $S^1_i$ from $q$ to $q'$, and in that case the
fixed points of $\bar{f}^p$ over $q$ and over $q'$ lie in different Nielsen classes.
\end{remark}

For $p \in \mathbb{N}^*$, let $\operatorname{Per}_{p, C_n}(f)$ denote the set of
periodic points of $f$ of minimal period $p$ that do not lie on $C_n$. Moreover, for simplicity let $Q_p:=1 - \operatorname{tr}\big((R(b)^{\pi_\mu})^p\big) + \operatorname{tr}\big((\bar{R}(b)^{\pi_\mu})^p\big)$.

\begin{theorem}\label{thmapp}
Let $p \in \mathbb{N}^*$ and assume that $f^p$ has no fixed points on $C_n$. Then
\[
|\operatorname{Per}_{p, C_n}(f)| \geq p\, M_p ,
\]
where $M_p$ is the number of monomials $t_1^{i_1} \cdots t_m^{i_m}$ with non-zero coefficient in
$Q_p$ such that $\gcd(p, i_1, \dots, i_m) = 1$.
\end{theorem}

\begin{proof}
By Proposition~\ref{prop:iterate}, the polynomial
$Q_p$
describes the fixed points $x$ of $\bar{f}^p$, where the coefficient of a monomial $t_1^{i_1} \cdots t_m^{i_m}$
is the index of the set $\operatorname{Fix}_I(\bar{f}^p)$ of fixed points of $\bar{f}^p$ with linking
numbers $I = (i_1, \dots, i_m)$. Thus, every monomial with a non-zero coefficient guarantees that
$\operatorname{Fix}_I(\bar{f}^p)$ is non-empty. Since $f^p$ has no fixed points on $C_n$,
Remark~\ref{extra} applied to $f^p$ shows that $\bar{f}^p$ has no fixed points on $T$, so
$\operatorname{Fix}(\bar{f}^p) = \operatorname{Fix}(f^p)$. Therefore these are fixed points of $f^p$ in
$B^3 \setminus C_n$, that is, periodic points of $f$ whose periods divide $p$.

We note that such a point $x$ might have a minimal period $k$ that is a proper divisor of $p$.
These are the points that we need to filter out. Over $p$ iterations, the trajectory of $x$ goes
around its orbit $p/k$ times, so by \eqref{eq:lk-iterate} with $q = p/k$,
\[
i_j = \ell k_{\bar{f}^p}(P_j, x) = \tfrac{p}{k}\, \ell k_{\bar{f}^k}(P_j, x), \quad \text{for all } j \in \{1, \dots, m\}.
\]
As a result, $p/k$ divides each of the exponents $i_1, \dots, i_m$. Since $p/k$ also divides $p$, then $p/k$ is a common divisor of $p, i_1, \dots, i_m$, and hence it divides
$\gcd(p, i_1, \dots, i_m)$. Therefore, the condition $\gcd(p, i_1, \dots, i_m) = 1$ forces
$p/k = 1$, that is, $k = p$. Hence each of the monomials in $M_p$ corresponds to points of minimal period $p$
under $f$.

Finally, we show that each of these monomials represents at least $p$ periodic points. By \eqref{eq:lk-orbit} with $k = p$, the $p$ distinct points $x, f(x), \dots, f^{p-1}(x)$ all lie in
$\operatorname{Fix}_I(\bar{f}^p)$, and none of them lies on $C_n$. 
Therefore, if a point $x$ corresponds to a given monomial, its entire
periodic orbit $\{x, f(x), \dots, f^{p-1}(x)\}$ corresponds to the same monomial. Since every fixed
point of $\bar{f}^p$ corresponds to exactly one monomial, distinct monomials are realized by
disjoint periodic orbits. Since every orbit of minimal period $p$ contains exactly $p$ distinct
points, none of which lies on $C_n$, we obtain at least $M_p$ disjoint orbits of $p$ points each,
and hence $|\operatorname{Per}_{p, C_n}(f)| \geq p\, M_p$.
\end{proof}

\begin{remark}
Note that the assumption of Theorem~\ref{thmapp} holds, for example, if no component of $C_n$
has period dividing $p$, a condition that can be determined from the permutation $\mu$. Without this assumption, $\bar{f}^p$
may have fixed points on the boundary tori $T$, and a single torus may contain several Nielsen
classes, so the monomials of $Q_p$ cannot all be connected to periodic points of $f$ in
$B^3 \setminus C_n$.
\end{remark}

\begin{remark}
The condition $\gcd(p, i_1, \dots, i_m) = 1$ in Theorem \ref{thmapp} is sufficient, and guarantees that the points of
$\operatorname{Fix}_I(\bar{f}^p)$ have minimal period $p$ under $f$, but it is not necessary. That is, a point of minimal period $p$
does not necessarily satisfy this condition. For
instance, a point of minimal period $2$ may have linking numbers $(2, 0)$ with respect to $f^2$, and thus $\gcd(p, i_1, i_2)=2\neq1$. By \eqref{eq:lk-iterate}, a fixed point of $f$ with linking numbers $(1, 0)$ also has linking
numbers $(2, 0)$ with respect to $f^2$, which gives $\gcd(p, i_1, i_2)=2\neq1$. Hence, from a monomial with
$\gcd(p, i_1, \dots, i_m) > 1$ alone we cannot decide whether it comes from a genuine period $p$ orbit, or just from a fixed point of $f$. Such monomials are not counted in $M_p$, so the
bound of Theorem~\ref{thmapp} remains valid, but it may underestimate
$|\operatorname{Per}_{p, C_n}(f)|$.
\end{remark}

\section*{Acknowledgement}

This project has received funding from the European Union’s Horizon 2024 research and innovation programme under the Marie Skłodowska-Curie grant agreement No. $101205038$. The author would also like to thank Jean-Baptiste Meilhan for his valuable input on the definition of the linking numbers, as well as the referee, whose careful remarks, comments and
corrections substantially improved the clarity and rigor of the present work.

\bibliographystyle{plain}
\bibliography{bibliLef}

\end{document}